\documentclass[12pt]{amsart}
\usepackage{amssymb,amsmath}
\usepackage{geometry}    
\usepackage{mathrsfs}

\usepackage{graphicx}

\usepackage{vmargin}
\setmargrb{1in}{1in}{1in}{1in}

\newtheorem{theorem}{Theorem}[section]
\newtheorem{lemma}[theorem]{Lemma}
\newtheorem{proposition}[theorem]{Proposition}
\newtheorem{corollary}{Corollary}
\theoremstyle{definition}
\newtheorem{definition}{Definition}
\newtheorem{remark}{Remark}

\newtheorem{problem}{Problem}

\newcommand{\scp}[1]{\langle#1\rangle}

\newcommand{\uz}{\underline{\zeta}}
\newcommand{\ux}{\underline{\xi}}

\newcommand{\Rm}{\mathbb{R}^m}

\newcommand{\ut}{\underline{\theta}}

\newcommand{\uy}{\underline{y}}

\begin{document}

\title[Exact uniform approximation and Dirichlet spectrum]{Exact uniform approximation and Dirichlet spectrum in dimension at least two}

\author{Johannes Schleischitz}


\thanks{Middle East Technical University, Northern Cyprus Campus, Kalkanli, G\"uzelyurt \\
    johannes@metu.edu.tr ; jschleischitz@outlook.com}


\begin{abstract}
   For $m\geq 2$, we determine the Dirichlet spectrum in $\Rm$ with respect to simultaneous approximation and the maximum norm 
   as the entire interval $[0,1]$. This complements previous
   work of several authors, especially Akhunzhanov and Moshchevitin, who considered $m=2$ and Euclidean norm.
   We construct explicit examples of real Liouville 
   vectors realizing any value in the unit interval. In particular, for positive values, they are neither badly approximable nor singular. Thereby we obtain a constructive proof of the main claim in a recent paper by Beresnevich, Guan, Marnat, Ram\'irez and Velani, who obtained a countable partition of $[0,1]$ into intervals with each having non-empty intersection with the Dirichlet spectrum. 
   Our construction is flexible enough to show that the according set of vectors with prescribed Dirichlet constant
   has large packing dimension and rather large 
   Hausdorff dimension as well.
   We establish a more general result on exact uniform approximation, applicable to a wide class of approximating functions.
   Our constructive proofs are considerably shorter and less involved than previous work on the topic. By minor twists of our proof, we infer similar, slightly weaker results when restricting to a certain class of classical fractal sets or other norms. In an Appendix we address the situation of a linear form.
\end{abstract}

\maketitle

{\footnotesize{

Math Subject Classification 2020: 11J06, 11J13}}



\section{Dirichlet spectrum}  \label{s1.1}

Let $\Vert x\Vert$ be the distance of $x\in \mathbb{R}$ to the nearest integer and for $\underline{x}\in \Rm$
let $\Vert \underline{x}\Vert= \max\{\Vert x_1\Vert, \ldots, \Vert x_m\Vert\}$. Given $\ux\in\Rm$, define
the non-increasing, piecewise constant, right-continuous function
\[
\psi_{\ux}(Q)= \min_{1\leq q\leq Q} \Vert q\ux\Vert, \qquad\qquad Q\geq 1,
\]
where $q$ ranges over the positive integers up to $Q$.
Let us then call
\begin{equation}  \label{eq:th}
\Theta(\ux):= \limsup_{Q\to\infty}\; Q^{1/m}\psi_{\ux}(Q),
\end{equation}
the Dirichlet constant of $\ux$, which is thereby
considered with respect to 
simultaneous approximation and the 
maximum norm. Define the Dirichlet spectrum $\mathbb{D}_m$ as the 
set of all values that the Dirichlet constant takes, i.e.
\[
\mathbb{D}_m=\{ \Theta(\ux): \ux\in \Rm \}.
\]
(Note: Occasionally, as in~\cite{am2},~\cite{as}, the $m$-th power of $\Theta(\ux)$ 
is taken, leading to an accordingly altered Dirichlet spectrum
$\mathbb{D}_m^m$). For the accordingly defined Lagrange spectrum when
considering the lower limit in \eqref{eq:th} instead, see~\cite{akm} for a very general result.
The set $\mathbb{D}_m$ is contained in the interval
$[0,1]$ by Dirichlet's Theorem. 
It was proved in~\cite{ds} that $\Theta(\ux)=1$ for Lebesgue almost 
all $\ux\in\Rm$,
see also the very recent paper by Kleinbock, Str\"ombergsson, Yu~\cite{ksw} 
for a considerably refined result and further references. 
For $m=1$, the Dirichlet spectrum is a rather complicated, well-studied object,
see~\cite{am2} for a wealth of references.
In particular, it is known
that $\mathbb{D}_1$ is not an interval, and contained in $\{0\}\cup[1/2,1]$ by a result
of Khintchine~\cite{khint}.  See further for example~\cite{feng},~\cite{huang},~\cite{huss},~\cite{kw} for refined
metrical claims when restricting to $\ux$ with $\Theta(\ux)=1$
when $m=1$. It is worth mentioning that
for $m=1$, Davenport and Schmidt~\cite{ds2} showed that, besides rational numbers, precisely numbers with bounded partial
quotients induce $\Theta(\xi)<1$. These coincide with
the set of badly approximable numbers for which 
$\liminf_{Q\to\infty} Q\psi_{\ux}(Q)>0$.
The claim is no longer
true for any $m>1$ and accordingly defined set of badly approximable vectors in $\Rm$ inducing $\liminf_{Q\to\infty} Q^{1/m}\psi_{\ux}(Q)>0$.
However, any badly approximable
vector is Dirichlet improvable in any dimension,
again a result due to Davenport and Schmidt~\cite[Theorem~2]{ds}.

For $m\geq 2$, the set of vectors that satisfy $\Theta(\ux)=0$, commonly referred to as
singular vectors, has Hausdorff dimension $m^2/(m+1)$, see~\cite{cheche}. Moreover, it is easy to see that
the $(m-1)$-dimensional set of vectors that are $\mathbb{Q}$-linearly dependent
together with $\{ 1\}$ shares this property.
Hence $\{0,1\}\subseteq \mathbb{D}_m$ for any $m\geq 2$. 
For $m=2$ and with respect to the Euclidean norm, results on the Dirichlet spectrum
were obtained by Akhunzhanov and Shatskov~\cite{as} and Akhunzhanov and Moshchevitin~\cite{am2}. In~\cite{as} it is shown
that this Dirichlet spectrum is an interval
which in some natural sense is
as large as it can be. For arbitrary 
norms, very recently structural results for $\mathbb{D}_2$ were obtained by Kleinbock and Rao~\cite{kr}, see also~\cite{kr2}. For $m\geq 2$ and
in the dual setting of a linear form in $m$ variables, some results
on the Dirichlet spectrum are immediate
from~\cite{beretc, marnat}, see Theorem~\ref{bt} in the Appendix. 
Also from~\cite{cheche} some metrical information can be inferred. None of these results 
however implies the existence of 
any non-empty interval where the Dirichlet spectrum with respect to
some norm is dense, when $m\geq 3$. In Corollary~\ref{rocky} we 
provide an interval contained in the Dirichlet spectrum, for a wide
class of norms on $\Rm$.

\section{Determination of Dirichlet spectrum in $\Rm$}  \label{se2.2}
We show that if $m\geq 2$, there exist (Liouville) vectors with any prescribed Dirichlet constant in $[0,1]$. In fact this set is rather
large in some metrical sense. 

\begin{theorem}  \label{A}
	Let $m\geq 2$. For any $c\in[0,1]$, there exists a
	set $\mathscr{A}_{m,c}\subseteq \Rm$ of packing dimension $m-1$ 
	consisting of $\ux\in\Rm$ satisfying
	\begin{equation} \label{eq:ii}
	\Theta(\ux) = c
	\end{equation}
	and for every $N$ we have
		\begin{equation} \label{eq:i}
\liminf_{Q\to\infty} Q^{N}\psi_{\ux}(Q) =0.
\end{equation}
In particular $\mathbb{D}_m=[0,1]$.
\end{theorem}

The claim is very much in line with the result for $m=2$ and the Euclidean norm quoted in~\S~\ref{s1.1}. Note that on the other hand that property \eqref{eq:i} forces the Hausdorff
dimension of $\mathscr{A}_{m,c}$ to be $0$ by 
Jarn\'ik-Besicovich Theorem~\cite{jarnik}.
See however Theorems~\ref{hdd},~\ref{beides} below for non-trivial
Hausdorff dimension results
when we drop hypothesis \eqref{eq:i}.
Theorem~\ref{A} is an immediate corollary
of the more general Theorem~\ref{F} below for a much larger class of
uniform approximating functions.

\begin{definition}
For $m\geq 2$ a fixed integer and $\Phi: \mathbb{N}\to (0,1)$ any
function, we define decay properties
$(d1), (d2), (d3)$ and for $\gamma>0$ the property $(d4(\gamma))$ as follows.

\begin{itemize}
	\item[(d1)] Assume
	\begin{equation*}
	\Phi(t)<t^{-1/m}, \qquad t\geq t_0.
	\end{equation*}
	\item[(d2)] Assume
	\[
	\lim_{t\to\infty} t^{\frac{1}{m-1} } \Phi(t)= \infty.
	\]
	\item[(d3)] Assume
	\[
	 \liminf_{\alpha\to 1^{+}} \; \liminf_{t\to\infty} \frac{\Phi(\alpha t)}{\Phi(t)}\geq 1, \qquad \text{if}\; m\geq 3,
	\]
	and
		\[
	\liminf_{\alpha\to 1} \; \liminf_{t\to\infty} \frac{\Phi(\alpha t)}{\Phi(t)}\geq 1, \qquad \text{if}\; m=2,
	\]
	where $t$ and $\alpha t$ are considered integers so that the
	expression is well-defined. 
		\item[$(d4(\gamma))$] Assume for given $\gamma>0$ and some 
		$\eta>0$, we have
			\[
		\Phi(t)> \eta t^{-\gamma}, \qquad t\geq t_0.
		\]
\end{itemize}

\end{definition}

An alternative formulation of $(d3)$ is that for every $\epsilon_0>0$
there is $\epsilon_1>0$, $t_0>0$ such that for any $\alpha\in(1,1+\epsilon_1)$ (resp. $\alpha\in(1-\epsilon_1,1+\epsilon_1)$ when $m=2$)
and $t\geq t_0$ we have 
\[
\frac{\Phi(\alpha t)}{\Phi(t)}\geq 1-\epsilon_0.
\]
Our more general result reads as follows.

\begin{theorem}  \label{F}
	Let $m\geq 2$ an integer and $\Phi$ satisfy $(d1), (d2), (d3)$. Then there exist uncountably many $\ux\in\Rm$ 
	for which the claims $(C1), (C2), (C3)$ below hold: 
	\begin{itemize}
		\item[(C1)] We have
	\[
	\psi_{\ux}(Q)< \Phi(Q), \qquad Q\geq Q_0.
	\]
	\item[(C2)]  For any $\varepsilon>0$, we have
	\begin{equation*}  
	\psi_{\ux}(Q) > (1-\varepsilon) \Phi(Q)
	\end{equation*}
	for certain arbitrarily large $Q$.
	\item[(C3)] Property \eqref{eq:i} holds for any given $N$.
\end{itemize}
	If for some $\gamma>0$ the function $\Phi$ satisfies $(d4(\gamma))$,
	then the packing dimension of the set of $\ux$ as above is at 
	least $m(1-\gamma)$.
\end{theorem}

We can always choose $\gamma=1/(m-1)$ 
by $(d2)$, however we require $\gamma\geq 1/m$ by $(d1)$, so the bound
always lies in the short iterval $[m-1-\frac{1}{m-1},m-1]$.
It is rather satisfactory and seems to exhaust the method,
it coincides with the estimate in~\cite[Theorem~2.1]{ichneu} where the larger sets of points singular of order
at least $\gamma$ are studied. 
If we assume $\Phi$ is decreasing, then we may relax $(d4(\gamma))$
by requiring its inequality only for certain arbitrarily large $t$.
Theorem~\ref{A} represents the special case $\Phi(t)=ct^{-1/m}$
if $c\in(0,1)$, which clearly satisfies $(d1), (d2), (d3), (d4(1/m))$, and
slightly altered functions in the special cases $c\in \{0,1\}$.
If we admit a factor $1+\varepsilon$ in the right hand side of $(C1)$
and are given an explicit rate of divergence in $(d2)$, from our proof we may give a rate for $Q$ in terms of $\varepsilon$
for which $(C1), (C2)$ hold.  We prefer to omit the details but
point out that similar results have been obtained
by Akhunzhanov~\cite{akm}
for ordinary approximation, i.e. demanding $(C1)$ only for certain arbitrarily large $Q$ but $(C2)$ for all large $Q$, and omitting $(C3)$.

We discuss the assumptions on $\Phi$.
Note that we do not require $\Phi$ to be decreasing.
Property $(d1)$ is very natural and necessary by Dirichlet's Theorem.
Conversely, condition $(d2)$ implies that $\Phi$ does not decay too fast.
It does not make sense to replace the exponent $1/(m-1)$ by 
a value larger than $1$ in view
of Khintchine's result~\cite{khint} quoted in~\S\ref{s1.1}.
Property $(d3)$ is very mild and in particular satisfied for all 
functions $\Phi(t)=ct^{-\tau}$, with $\tau>0, c>0$. It ensures
that $\Phi$ does not decay (and not rise when $m=2$) very fast in short intervals.
From $(d1), (d3)$ we see that in fact the lower limit for $\alpha$ in $(d3)$ must equal $1$. 
As observed above, $(d2)$ implies $(d4(\frac{1}{m-1}))$.
We further remark that claim $(C2)$ and
property $(d2)$ imply 
that the coordinates of $\ux$ in Theorem~\ref{F} together with $\{1\}$ are linearly independent over $\mathbb{Q}$, in other words
$\ux$ is totally irrational. We go on to comment on potential relaxations/removal of the conditions $(d2), (d3)$ in \S~\ref{nee}. 

Denote by $Bad_m$ the set of badly approximable vectors in $\Rm$
as introduced in \S~\ref{s1.1}.
Our claim \eqref{eq:i} means that $\ux$ in Theorem~\ref{F} are Liouville vectors and hence clearly not badly approximable. 
Write
\[
Di_m(c)= \{ \ux\in\Rm: \psi_{\ux}(Q)\leq cQ^{-1/m}, \; Q\geq Q_0 \}
\subseteq \{ \ux\in\Rm: \Theta(\ux)\leq c \},  
\]
so that $Di_m=\cup_{c<1} Di_m(c)$ is the set of 
$m$-dimensional Dirichlet improvable vectors. Further denote by $Sing_m$ 
the set of singular vectors in $\Rm$, defined via the property $\lim_{Q\to\infty} Q^{1/m}\psi_{\ux}(Q)=0$, or 
equivalently $\cap_{c>0} Di_m(c)$. The next corollary 
of Theorem~\ref{F} slightly
refines Theorem~\ref{A}.

\begin{corollary}  \label{c}
	Let $m\geq 2$ be an integer. For any $c\in(0,1]$, the set
\[
Di_m(c) \setminus (\cup_{\epsilon>0} Di_m(c-\epsilon)\cup Bad_m)
\]
	has packing dimension at least $m-1$. 
	In particular, the same applies to the
	set
	\[
	\boldsymbol{FS}_m:= Di_m\setminus (  Bad_m \cup Sing_m  ).
	\]
\end{corollary}

The latter claim extends the main result from~\cite{beretc} in two
directions. Firstly, we also
give a metrical result instead of only proving uncountability,
thereby contributing towards the metrical problem of
determining the Hausdorff dimension of $\boldsymbol{FS}_m$ formulated in~\cite[\S~3.4]{beretc}. 
Asymptotically as $c\to 0$,
our metrical bound is probably sharp up to an additive error $O(m^{-1})$ in view of claims by
Cheung and Chevallier~\cite{cheche} implying
that the set $Di_m(c)$ 
has Hausdorff dimension $m-1+\frac{1}{m+1}+o(1)$ as $c\to 0$, with
positive error term for any $c>0$.
In fact the same conclusion holds for 
the set $Di_m(c)\setminus Di_m(\delta)$
with some explicitly computable $\delta=\delta(c)\in (0,c)$
and $c$ small enough, however
results from~\cite{cheche} do not allow for taking $\delta$ arbitrarily close to $c$. 
The same should be expected for packing dimension as well, 
as suggested by the results in~\cite{dfsu1, dfsu2}.
Note also that as follows from Kleinbock and Mirzadeh~\cite[Theorem~1.5]{km},
the set $Di_m(c)$ has Hausdorff dimension less than $m$ for any $c<1$.
On the other hand, it is conjectured in~\cite[Problem~3.1]{beretc} that
$\boldsymbol{FS}_m$ has full Hausdorff dimension. 

Secondly, we emphasize that by the first claim we can
also prescribe an exact Dirichlet constant (in the simultaneous approximation setting). Note
that in~\cite{beretc, marnat} the deep, unconstructive result of
Roy~\cite{roy} on parametric geometry of numbers was used.
As a consequence of this setup, 
from~\cite[Theorem~1.5]{beretc} and~\cite{marnat}, one can
only provide a countable partition of $[0,1]$ into intervals with
each having non-empty intersection with $\mathbb{D}_m$, 
see the Appendix of the paper. 
We should however remark that
additional specifications on various exponents
of approximation within these sets $\boldsymbol{FS}_m$
can be made according to~\cite{beretc, marnat}.
We should also note that the dual setting of a linear
form in $m$ variables is
treated in~\cite{beretc}. On the other hand,
our constructive proof of Theorem~\ref{F} is elementary and rather short, based on ideas from
the proof of~\cite[Theorem~2.5]{ichostrava}, with some twists.

We further provide a considerably weaker bound regarding Hausdorff dimension, which we will denote by $\dim_H$, however still of order $\gg m$. We consider slightly
larger sets than in Corollary~\ref{c}, namely
for $m\geq 2$ an integer and $c\in[0,1]$, our focus is now on the sets
\[
\textbf{F}_{m,c}:=\bigcap_{\epsilon>0} (Di_m(c+\epsilon) \setminus Di_m(c-\epsilon))\setminus Bad_m=
\{ \ux\in\Rm: \Theta(\ux)= c \}  \setminus Bad_m\subseteq \boldsymbol{FS}_m.
\]  
We show

\begin{theorem}  \label{hdd}
	For any $m\geq 2$ and $c\in[0,1]$, we have
	\begin{equation} \label{eq:sinus2}
	\dim_H(\textbf{F}_{m,c})\geq 
	\frac{ \sqrt{m(m^2-m+1)} }{  \left( \frac{m+\sqrt{m(m^2-m+1)}}{m-1}  \right)^2 + (m-\frac{1}{m})\left( \frac{m+\sqrt{m(m^2-m+1)}}{m-1}  \right) -1  } > 0.
	\end{equation}
	Asymptotically as $m\to\infty$, uniformly in $c\in(0,1]$ we have the stronger lower bound
	\begin{equation} \label{eq:cosinus2}
	\dim_H(\textbf{F}_{m,c})\geq \frac{3}{8}m-o(m).
	\end{equation}
\end{theorem}

The estimate \eqref{eq:sinus2} is only of order $1-o(1)$ as $m\to\infty$,
so the latter asymptotical bound \eqref{eq:cosinus2} 
is indeed significantly stronger. It is natural to expect that
$c\longmapsto \dim_H(\textbf{F}_{m,c})$ decays, which is not reflected
in our result. In contrast to Theorem~\ref{F},
our estimates \eqref{eq:sinus2}, \eqref{eq:cosinus2} 
are probably far from the true value no matter how
small $c>0$ is chosen. Indeed, it is reasonable to conjecture
$\dim_H(\textbf{F}_{m,c})= \dim_H(Di_m(c))$ for any $c\in(0,1]$, in particular 
\[
\dim_H(\textbf{F}_{m,c})= m-1+\frac{1}{m+1}+o(1), \qquad \text{as} \; c\to 0,
\]
and
\[
\dim_H(\textbf{F}_{m,c})= m-o(1), \qquad \text{as} \; c\to 1,
\] 
see the comments below Corollary~\ref{c}. See further~\cite[Theorem~2.2]{ichneu} for a
stronger bound of order $m-4+O(m^{-1})$ for the Hausdorff dimension of
the larger set of (inhomogeneously) singular vectors obtained from 
essentially the same method in a simplified setting. 
Recall there was no such discrepancy to~\cite{ichneu} for 
the packing dimension
result, as remarked below Theorem~\ref{F}.
Theorem~\ref{hdd}
can be generalized to the situation of $\Phi(t)$ as in Theorem~\ref{F} satisfying $(d1), (d2), (d3), (d4(\gamma))$
for some $\gamma>0$, we do not explicitly state it.

We next derive a theorem that contains information on ordinary
approximation as well, very much in the spirit of~\cite{beretc, marnat}. Let $\lambda(\ux)$ denote the ordinary exponent
of simultaneous rational approximation to $\ux\in\Rm$, defined as the supremum of $\lambda>0$ such that
\[
\liminf_{Q\to\infty} Q^{\lambda} \psi_{\ux}(Q) = \liminf_{q\to\infty}
q^{\lambda}\Vert q\ux\Vert
< \infty.
\]
Then $\lambda(\ux)\in [1/m,\infty]$ for any $\ux\in\Rm$ by Dirichlet's Theorem.
Denote by
\[
\mathcal{W}_m(\lambda)= \{ \ux\in\Rm: \lambda(\ux)= \lambda \}\subseteq \Rm
\]
the pairwise disjoint levelsets of vectors with precise 
ordinary exponent $\lambda\in [1/m,\infty]$.  Note that $Bad_m\subseteq \mathcal{W}_m(\frac{1}{m})$. Let $\beta=\frac{1+\sqrt{5}}{2}$ be the golden ratio.

\begin{theorem}  \label{beides}
	Let $m\geq 2$ be an integer, and $c\in(0,1)$ and 
	$\lambda\in ( \beta, \infty)$. Then the Hausdorff dimension
	of the set
	\begin{equation}  \label{eq:iset}
	\mathcal{W}_m(\lambda) \cap \textbf{F}_{m,c}\; \subseteq \; 
	\mathcal{W}_m(\lambda) \cap \boldsymbol{FS}_m
	\end{equation}
	is positive and a lower bound independent of $c$ is explicitly computable.
	Asymptotically
	as $\lambda\to \infty$, i.e. for $\lambda\geq \lambda_0(m)$,
	it is of order
	\[
	\dim_H(\mathcal{W}_m(\lambda) \cap \textbf{F}_{m,c}) \geq 
	\frac{m}{2\lambda} - O(\lambda^{-1}),
	\]
	where the implied constant is effectively computable and
	does not depend on $m, c, \lambda$.
\end{theorem}

By Jarn\'ik-Besicovich Theorem~\cite{jarnik}, we have $\dim_H(\mathcal{W}_m(\lambda))=(m+1)/(\lambda+1)$, so 
for large $m$ our 
asymptotical bound is basically sharp
up to a factor $2$. 

We may again extend the 
claim to a setup involving in place of $\textbf{F}_{m,c}$ sets derived
from more general uniform approximation functions $\Phi$ via 
imposing $(C1), (C2)$. 
Besides, with small modifications in the proof below, we 
can prescribe the order of ordinary approximation more exactly up to
an asymptotical factor $1+o(1)$ as $Q\to\infty$. 
More precisely, take any function $\Psi: \mathbb{N}\to (0,1)$ of decay $o(t^{-\beta-\epsilon})$ as $t\to\infty$ for some $\epsilon>0$,
and conversely satisfying $(d4(\gamma))$ for some $\gamma>0$.
Derive $\mathcal{W}_m(\Psi)$ the set of $\uz\in\Rm$ satisfying
\[
\liminf_{Q\to\infty} \frac{ \psi_{\uz}(Q)}{\Psi(Q)}=
\liminf_{q\to\infty} \frac{\Vert q\uz\Vert}{\Psi(q)} =1.
\]
Then $\mathcal{W}_m(\Psi)\cap \textbf{F}_{m,c}$ has positive Hausdorff dimension, effective lower bounds can be given subject to 
the decay rate of $\Psi$. The special case $\Psi(t)=\Psi_{a,\lambda}(t):=at^{-\lambda}$ for $\lambda>\beta$ and $a>0$ parameters turns out 
to result in the same bounds as Theorem~\ref{beides} (see~\S~\ref{gen}), thereby 
refining it since $\mathcal{W}_m(\Psi_{a,\lambda})\subsetneq \mathcal{W}_m(\lambda)$. We sketch the proof of this generalized claim in~\S~\ref{gen}, but
want to compare this version of
Theorem~\ref{beides}
to~\cite{marnat}. There it was shown that there is 
some explicitly computable 
$\kappa_m\in (0,1)$, so that for any $\lambda>1/m$
and fixed $c\in(0,1]$, the set of $\uz\in\Rm$  
inducing simultaneously
\[
\liminf_{Q\to\infty} \frac{\psi_{\uz}(Q)}{\Psi_{a,\lambda}(Q)}\in[\kappa_m,1],\qquad \limsup_{Q\to\infty} \frac{\psi_{\uz}(Q)}{cQ^{-1/m}}\in[\kappa_m,1],
\]
is uncountable. Our claim allows for prescribing the order of 
both ordinary and uniform approximation considerably more precisely, provides a metric claim, and 
permits more flexibility in the choice of functions $\Psi$ locally,  
for the cost of requiring a faster decay rate for $\Psi$.

The lower bound $\beta$ for $\lambda$ can in fact be improved to $\lambda>1$ with a rather technical argument, 
we prefer to only sketch the proof in~\S~\ref{gen} below.
However, that seems to be the limit of the method. On the other hand,
we strongly expect the Hausdorff dimension of the sets
in \eqref{eq:iset} to decay
as a function of $\lambda\geq 1/m$ (and of $c$ as well). The above remarks on more general $\Psi$ still apply for any $\Psi(t)=o(t^{-1-\epsilon})$.

\section{Cantor sets and other norms}

\subsection{Dirichlet spectrum for Cantor sets}  \label{fra}

For $b\geq 2$ an integer and $R\in (0,1)$ parameters,
we define a modified decay property $(d3^{\prime}(b,R))$ and a Diophantine property $(D(b))$ on functions $\Phi: \mathbb{N}\to (0,1)$.

\begin{definition} Let $b\geq 2$ an integer and $R\in (0,1)$. We define
	\begin{itemize}
		\item[$(d3^{\prime}(b,R))$] We have
		\[
		\Phi(bt) > R\Phi(t), \qquad t\geq t_0.
		\]
		\item[(D(b))] For any fixed $\epsilon>0$, the inequalities
			\begin{equation*}  
		(1-\epsilon)\Phi(b^B) < b^{-A} < \Phi(b^B) 
		\end{equation*}
		hold for certain arbitarily large pairs of 
		positive integers $A,B$.
	\end{itemize}
\end{definition}

Property $(d3^{\prime}(b,R))$ relaxes $(d3)$.
The Diophantine property $(D(b))$ is rather mild and applies
to most reasonable functions, however with the unfortunate
exception of functions
$\Phi(t)=ct^{-r/s}$ for a rational number $r/s$ and $c>0$.

For $b\geq 2$ an integer and $W\subseteq \{ 0,1,\ldots, b-1\}$
of cardinality $|W|\geq 2$, define the Cantor set $C_{b,W}$ as the
set of all real numbers that admit a base $b$ representation 
\[
\sum_{j=1}^{\infty} w_jb^{-j}, \qquad w_j\in W.
\]
The Cantor middle
third set just becomes $C_{3,\{0,2\}}$.
Let $W_1,\ldots, W_m$ be arbitrary sets $W$ as above to a uniformly chosen base $b$, and define the Cartesian product set
\[
K= \prod_{i=1}^{m} C_{b, W_i}. 
\]
In the special case $\{0,1\}\subseteq W_i$ for all $1\leq i\leq m$ we say $K$ is good.
We claim that in the assertions of~\S~\ref{se2.2}, we may restrict to
$\ux\in K$ upon taking a smaller, fixed multiplicative constant $R<1$ in $(C2)$ and adjusting the metrical claims. 
Assuming some Diophantine condition,
we even get a precise analogue of Theorem~\ref{F}. 
Let us first define some more properties.
Denote by $\Gamma_b=lcm(1,2,\ldots,b-1)$ the least common multiple of the positive integers up to $b-1$ and by $\dim(K)= \sum_{i=1}^{m} \log |W_i|/\log b$ the
Hausdorff (or packing) dimension of $K$.

\begin{theorem}  \label{2}
	Let $m\geq 2$ and $K$ be any set as above. 
	Assume $\Phi$ satisfies $(d1), (d2)$. 
	\begin{itemize}
		\item[(i)] 
	Assume $\Phi$ also satisfies $(d3^{\prime}(b,R))$ 
	for some given $R\in (0,1)$. Let 
\[
\Omega=\Omega(b,R)= \begin{cases}
R, \qquad\qquad\qquad\qquad\quad \text{if}\; $K$ \; \text{is good}, \\
R^{b+2}(b-1)^{-3}\Gamma_{b}^{-1}, \qquad\; \text{otherwise}. 
\end{cases}
\]
	Then there exist uncountably many
	$\ux\in K$ for which we have $(C1), (C3)$ and 	
	\begin{itemize}
		\item[$(C2^{\prime})$] We have
		\begin{equation*} 
		\psi_{\ux}(Q) > \Omega\Phi(Q)
		\end{equation*}
		for certain arbitrarily large $Q$.
	\end{itemize}
	\item[(ii)] Assume $\Phi$ satisfies $(D(b))$ as well and $K$ is good. Then
	there exist uncountably many $\ux\in K$ satisfying $(C1), (C2)$
	and $(C3)$.
\end{itemize}

	If $K=C_{b,\{0,1\}}^m$ and for some $\gamma>0$ the function $\Phi$ satisfies $(d4(\gamma))$,
then the packing dimension of the set of $\ux$ in (i), (ii) is at 
least $\dim(K)(1-\gamma)=m (1-\gamma)\log 2/\log b$. Moreover, 
the set of vectors satisfying
$(C1^{\prime})$ and $(C2^{\prime})$ resp.
$(C1^{\prime})$ and $(C2)$ in (i) resp. (ii), 
has positive Hausdorff dimension, where $(C1^{\prime})$ is $(C1)$ upon admitting an additional factor $1+\varepsilon$.
\end{theorem}

\begin{remark}  \label{reh}
	As $b\to\infty$, we have $\Gamma_{b}=e^{(1+o(1))(b-1)}$ by Prime Number Theorem. It thus follows from our proof that the bound for $\Omega$ can be improved asymptotically as $b\to\infty$ by replacing numerator
	in the ``otherwise'' formula
	 by $R^{b/\log b+O(1)}$. Moreover, $\Omega$ can be significantly improved  
	when $m$ is small compared to $b$. In particular we can always
	replace $\Gamma_b$ by the maximum number that appears
	as $lcm$ of any $m$ positive integers at most $b-1$, hence 
	a crude upper estimate is given by $(b-1)^m$.
\end{remark}

We will show that when $\Phi(q)=cq^{-1/m}$, property $(d3^{\prime}(b,R))$ holds for $R=b^{-1/m}$.
Hence we get the following variant of Corollary~\ref{c}. 

\begin{corollary}  \label{ok}
Let $m\geq 2$ and $K$ as above. Let 
\[
\sigma=\sigma(m,b)= \begin{cases}
b^{-1/m}, \qquad\qquad\qquad\qquad \text{if}\; K \;\text{is good}, \\
b^{-(b+2)/m}(b-1)^{-3}\Gamma_b^{-1}, \quad \text{otherwise}.
\end{cases}
\]
Then $\sigma\in(0,1)$ and for any $c\in(0,1]$ the set
	\[
	K\cap (Di_m(c)\setminus (Di_m(\sigma c) \cup Bad_m)) \subseteq K\cap \boldsymbol{FS}_m
	\]
	is uncountable. If $K=C_{b,\{0,1\}}^m$ it 
	has packing dimension at least $\dim(K)(1-1/m)=(m-1)\log 2/\log b$ and positive Hausdorff dimension.
	 In particular the same holds for
	$K\cap \boldsymbol{FS}_m$.
\end{corollary}

We may calculate an effective positive lower bound 
for the Hausdorff dimension as well with the method of~\S~\ref{put},
however the expression becomes rather cumbersome.
Moreover, a variant of Theorem~\ref{beides} can be derived,
we do not state it.
Corollary~\ref{ok} induces a countable partition of $[0,1]$ into 
intervals that each have
non-empty intersection with $\mathbb{D}_m\cap K$, similar 
as in~\cite{beretc, marnat} (or ~\cite{cheche}), see Theorem~\ref{bt} in the Appendix of the paper, but restricting to fractal sets.
We present a class of functions that satisfy $(D(b))$ 
for any $b\geq 2$ and thus the 
stronger claim of Theorem~\ref{2} for good $K$.

\begin{corollary}  \label{rok}
	Assume $K$ is good.
	 Let $\Phi(q)=cq^{-\tau}$ for irrational $\tau\in(\frac{1}{m}, \frac{1}{m-1})$
	 and any $c>0$.
	 Then $\Phi$ satisfies $(d1), (d2), (D(b))$ and thus the set
	 of vectors $\ux\in K$ 
	 satisfying $(C1), (C2)$ and $(C3)$
	 	has packing dimension at least $\dim(K)(1-\tau)=\sum_{i=1}^{m} \log |W_i|(1-\tau)/\log b$, and if we drop $(C3)$ also positive Hausdorff dimension.
\end{corollary}

Unfortunately, the function $\Phi(t)=ct^{-1/m}$ does not satisfy 
hypothesis $(D(b))$ for any $b\geq 2$, therefore we cannot conclude that $\mathbb{D}_m\cap K=[0,1]$. We want to remark
that for ordinary simultaneous approximation and fast enough decaying $\Phi$, similar results on exact approximation to successive powers of elements of Cantor sets $C_{b,W}$, thereby
restricting to the Veronese curve,
have been
obtained in~\cite[\S 2.2]{ichmjnt}.

The sets $K$ in this section can be obtained as the attractor
of an iterated function system (IFS) consisting of a finite set of contracting maps
$f_j(\underline{x})= \underline{x}/b+ \underline{u}_j/b$ 
on $\Rm$ with integer
vectors $\underline{u}_j\in \mathbb{Z}^m$. We ask whether our results
extend to more general situations.

\begin{problem}
	Let $K\subseteq \Rm$ be any uncountable attractor of an IFS. Is the
	set $K\cap \boldsymbol{FS}_m$
	non-empty? Is it at least true for any IFS consisting of 
	contracting functions of the form
	$f_j(\underline{x})= A_j\underline{x}+\underline{b}_j$ with
	$A_j\in\mathbb{Q}^{m\times m}, \underline{b}_j\in\mathbb{Q}^m$?
\end{problem}

\subsection{Other norms}

Let us call a norm 
$|.|$ on $\Rm$ expanding
if $|\underline{x}|\geq |\pi_j(\underline{x})|$ for all $\underline{x}\in\Rm$ and $1\leq j\leq m$, 
where $\pi_j$ are the orthogonal projections to the coordinate axes.
Then

\begin{corollary}  \label{rocky}
	Let $|.|$ be any expanding norm on $\Rm$. Let $\underline{e}_i=(0,0,\ldots,0,1,0,\ldots,0)$, $1\leq i\leq m$, be the canonical base vectors and let $\chi:= \min_{1\leq i\leq m} |\underline{e}_i|>0$. Then $[0,\chi]\subseteq \mathbb{D}_{m}^{|.|}$
	where $\mathbb{D}_{m}^{|.|}$ is the Dirichlet spectrum with respect to $|.|$.
\end{corollary}

The result in particular applies to any $p$-norm 
$|\underline{x}|_p=(\sum |x_i|^p)^{1/p}$ and shows that
the interval $[0,1]$ is contained
in the according Dirichlet spectrum. This may be compared with the results for $m=2$ and Euclidean norm $p=2$ 
by Akhunzhanov, Shatoskov~\cite{as} and Akhunzhanov, Moshchevitin~\cite{am2} recalled in~\S~\ref{s1.1}. 
The deduction of Corollary~\ref{rocky} relies on the fact that for the real vectors $\ux$ constructed in~\S~\ref{proof}, for any $Q$ inducing large values of $Q^{1/m}\psi_{\ux}(Q)$ we may choose
$q<Q$ so that some component $\Vert q\xi_i\Vert$ of our choice 
induces a small value of $\Vert q\ux\Vert$ and 
significantly outweighs all the other $\Vert q\xi_j\Vert, j\neq i$.
This will give us any desired upper bound in $[0,\chi]$ 
for the Dirichlet constant.
The reverse lower bound will be immediate
from the assumption of the norm being expanding and our results
for the maximum norm.
We provide more details in~\S~\ref{crproof}.

\section{Some remarks and forthcoming work}

Our results can be interpreted as prescribing 
extremal values (primarily maxima, but also minima in Theorem~\ref{beides}) of the first successive minimum of some
classical parametric lattice point problem, see for example~\cite{ss},
in particular~\cite[Theorem~1.4]{ss}. 
Our method relies heavily on ideas from the proof of~\cite[Theorem~2.5]{ichostrava}, which contains a 
considerably good description of higher successive minima functions
as well.
Thus it seems plausible that similar results on extremal values
of the according parametric higher successive minima functions can be obtained when
combining our proofs below with some more ingredients from the proof of~\cite[Theorem~2.5]{ichostrava}. 

By transference inequalities, it is possible
to derive from our Theorem~\ref{A} some
information on the Dirichlet spectrum of a linear form in $m$ variables
as investigated in~\cite{beretc, marnat}.
We give details and compare our result 
with~\cite{beretc, marnat} in the Appendix of the paper.

Analogous results of Theorems~\ref{A},~\ref{F} for (a) $p$-adic approximation, 
(b) weighted approximation and/or (c) systems of linear forms, seem in reach by refinements and generalizations of the method presented below. 
Verification of (a), (b), (c) seems increasingly challenging, in particular the impression of the author is that some non-trivial new concepts need to be introduced for (c). The question of a weighted version
for linear forms, i.e. (b) and (c) together, was raised in~\cite[Problem~4.1]{beretc}. 
The author plans subsequent work on these topics.

\section{Proof of Theorem~\ref{F} for $m\geq 3$: Existence claim}

\subsection{Construction of suitable vectors $\ux$}  \label{ko}
Define an increasing sequence of positive integers
recursively as follows: For the initial terms, observe that 
by $(d2)$, there is an integer $H$ so 
that for any $Q\geq H$ we have 
$\Phi(Q) > Q^{-1/(m-1)}$.
Define the initial $m$ terms by $a_j=H^j$ 
for $1\leq j\leq m$.
For $n\geq 1$, having constructed the first $mn$ terms $a_1,\ldots,a_{mn}$, let
the next $m$ terms be given by the recursion
\begin{equation}  \label{eq:h}
a_{mn+1}= a_{mn}^{M_n}
\end{equation}
and
\begin{equation}  \label{eq:tats}
a_{mn+2}= a_{mn+1}^2, \quad a_{mn+3}=a_{mn+1}^3, \quad \ldots, \quad a_{mn+m-1}=a_{mn+1}^{m-1},
\end{equation}
and finally
\[
a_{mn+m}=a_{m(n+1)}= L_n\cdot a_{mn+m-1}\in (a_{mn+m-1}, a_{mn+1}^m]
\]
with integers $M_n\to\infty$ that tend to infinity fast enough, to be made precise below, 
and the integer
$L_n$ defined as 
\[
L_n=  \max\{ z\in\mathbb{N}: a_{mn+1}^{-1} < \Phi(Q),\; 1\leq Q\leq za_{mn+m-1} \}.
\]
By $(d1)$ we have
$\Phi(t)\to 0$ as $t\to\infty$ and hence $L_n$ is well-defined.
In fact $(d1)$ implies
\begin{equation} \label{eq:T}
L_n\leq a_{mn+1}, \qquad\qquad n\geq 1,
\end{equation}
so that indeed $a_{mn+m}\leq a_{mn+1}^m$ for $n\geq 1$.
Conversely,
by assumption $(d2)$ and our choice of initial terms,
we have $L_n\geq 1$ for all $n\geq 1$ and
\begin{equation}  \label{eq:toin}
\lim_{n\to\infty} L_n= \infty.
\end{equation}
In particular 
indeed $(a_j)_{j\geq 1}$ is strictly increasing.
We observe further that
\begin{equation} \label{eq:divid}
a_j|a_{j+1}, \qquad j\geq 1.
\end{equation}
Define the components $\xi_i$
of $\ux$ via
\begin{equation}  \label{eq:sinus}
\xi_i= \sum_{n=0}^{\infty} \frac{1}{a_{mn+i}}, \qquad 1\leq i\leq m,
\end{equation}
that is we sum the reciprocals
over the indices congruent to $i$ modulo $m$.
We claim that 
all assertions $(C1), (C2), (C3)$ of the theorem hold for $\ux=(\xi_1,\ldots,\xi_m)$
if we choose $M_n$ suitably large in each step.
We remark that taking a power of $a_{mn}$ in
\eqref{eq:h} is just for convenience, it suffices to let $a_{mn+1}=M_n^{\prime}a_{mn}$
for large enough $M_n^{\prime}$.  We will assume
the latter occasionally. We conclude this
section with a very elementary
observation to be applied below.

\begin{proposition}  \label{uv}
	Let $u/v\in\mathbb{Q}$ be reduced. Then for any integers 
	$m\geq 2, L\geq 1, M\geq 2$, the fraction
	\[
	\frac{u}{v}+ \frac{1}{(v^{m-1}L)^M}= \frac{uL^Mv^{(m-1)M-1} +1 }{L^Mv^{(m-1)M}}
	\]
	is reduced as well.
\end{proposition}

Obviously the numerator is congruent to $1$ modulo any prime
divisor of $Lv$, and the claim follows.

\subsection{Proof of $(C1), (C2), (C3)$}  \label{proof}
Proof of (C3): Take $Q=q=a_{mn}$ for $n$ large. 
By \eqref{eq:divid} all $qa_j^{-1}$ for $j\leq mn$ are integers.
Hence $\Vert q\xi\Vert=\Vert q\xi_1\Vert=a_{mn}(a_{mn+1}^{-1}+a_{m(n+1)+1}^{-1}+\cdots)$. The main contribution clearly comes from
the first term $a_{mn}/a_{mn+1}$, indeed we may estimate
\begin{equation}  \label{eq:oto}
\Vert q\ux\Vert= \Vert q\xi_1\Vert \leq 2\frac{a_{mn}}{a_{mn+1}}= 
2{a_{mn}}^{-M_n+1}= 2Q^{-M_n+1}.
\end{equation}
Since we assume $M_n\to\infty$ it suffices to take $n$ large enough
for given $N$.

Proof of (C1): 
For simplicity write
\[
d_n=a_{mn+1}, \qquad n\geq 1.
\]
Let $Q>1$ be an arbitrary large number and $k$ be the index with
$a_k\leq Q<a_{k+1}$. 
Write $k=mf+g-1$ for integers $f\geq 0$ and $g\in\{1,2,\ldots,m\}$.
We consider three cases.

Case 1: Assume $g=1$. This means $a_{mf}\leq Q<a_{mf+1}$.
Let $q=a_k=a_{mf}\leq Q$. Then very similar as in \eqref{eq:oto} we see
\[
\Vert q\ux\Vert= \Vert q\xi_1\Vert \leq 2\frac{a_{mf}}{a_{mf+1}}= 
2{a_{mf}}^{-M_f+1}= 2q^{-M_f+1}\leq 2Q^{-\frac{M_f-1}{M_f} }=
2Q^{-1+\frac{1}{M_f} }.
\]
Now since we can assume $M_f\geq 3$, from $(d2)$ we easily see that the right hand 
side is less than $\Phi(Q)$ for sufficiently large $Q\geq Q_0$
or equivalently $f\geq f_0$. Thus $\psi_{\ux}(Q)<\Phi(Q)$ for
$Q$ in these intervals.

Case 2: Assume $2\leq g\leq m-1$. Then $Q<a_{k+1}\leq a_{mf+m-1}= d_f^{m-1}$ and $a_{k+1}/a_k=d_f$. 
Let $q=a_{mf+1}=d_f\leq Q$. Isolating the first term
$a_k/a_{k+1}=d_f^{-1}$ in $q\xi_2$, we easily verify
\begin{equation}  \label{eq:et}
\Vert q\ux\Vert= \Vert q\xi_2\Vert \leq \frac{1}{d_f}+ O(d_{f+1}^{-1})\leq  2Q^{-1/(m-1)}< \Phi(Q)
\end{equation}
for $Q\geq Q_0$ or equivalently $f$ large enough by property $(d2)$. 
This again implies $\psi_{\ux}(Q)<\Phi(Q)$ for the 
$Q$ in question.

Case 3: Now assume $g=m$ or equivalently $k\equiv -1\bmod m$. 
By construction of $L_f$ and since $a_{k+1}=a_{m(f+1)}=L_fa_{mf+1}=L_fd_f$, we
have
\[
d_f^{-1} < \Phi(Q) , \qquad 1\leq Q\leq a_{k+1}.
\]
With $q=d_{f}\leq Q$ again, as in \eqref{eq:et} we have
\begin{equation}  \label{eq:te}
\Vert q\ux\Vert= \Vert q\xi_2\Vert \leq \frac{1}{d_f}+ O(d_{f+1}^{-1}).
\end{equation}
Combining, we derive the estimate
\begin{align*}
\Vert q\ux\Vert= \Vert q\xi_{2}\Vert < \Phi(Q)
\end{align*}
as well if we choose $M_{f+1}$ and hence $d_{f+1}=a_{m(f+1)+1}$ in the next step large enough that the error
term in \eqref{eq:te}
is small enough. Again $\psi_{\ux}(Q)<\Phi(Q)$ for
$Q$ in these intervals  follows. Since we covered all large numbers
with our cases, $(C1)$ follows. 

\begin{remark} \label{reee}
	Note that Case 3 is the critical point where 
	we needed the assumption
	$m>2$. Indeed, for $m=2$ we would not
	have $a_{2(f+1)}=a_{2f+2}=a_{2f+1}^2$ but rather $a_{2f+2}=L_fa_{2f+1}$, and
	for $\Phi$ slowly decaying, like $\Phi(t)=ct^{-1/2}$, the outcome 
	$\Vert q\ux\Vert= \Vert q\xi_2\Vert \leq 1/L_f+ O(d_{f+1}^{-1})$ would
	be larger than the bound in \eqref{eq:te}. We further remark that
	in Case 2 we could take $q=a_k$ instead and obtain $\Vert q\ux\Vert= \Vert q\xi_g\Vert\leq d_f^{-1}+O(d_{f+1}^{-1})$. 
\end{remark}

Proof of (C2): For some large integer $f$ let 
\begin{equation} \label{eq:del}
Q= Q_f= a_{m(f+1)}-1.
\end{equation}
Let $1\leq q\leq Q$ be any integer and $s=s(q)$ be the maximum index
with $a_s$ divides $q$, and take $s=0$ and 
let $a_0=1$ if no such $s$ exists. Recall $d_f=a_{mf+1}$.
We claim that
\begin{equation}  \label{eq:claim}
\Vert q\ux\Vert \geq \frac{1}{d_{f}}+ O(Qd_{f+1}^{-1}).
\end{equation}
We again distinguish two cases.

Case I: We have $s\leq mf$, or equivalently $a_s\leq a_{mf}$. Then we claim
\begin{equation} \label{eq:fre}
\Vert q\ux\Vert \geq \Vert q\xi_1\Vert \geq d_f^{-1} + O(Qd_{f+1}^{-1}).
\end{equation}
Write $e_s=e_s(q)= a_{s+1}/a_s\in\mathbb{Z}$ for simplicity,
which is an integer by \eqref{eq:divid}. By assumption we have $q=B\cdot a_{s}$ with some integer $B=B_q$ with $e_s\nmid B$ and $B=q/a_s\leq Q/a_s < a_{k+1}/a_s$. 
We split $q\xi_1= U_q+V_q$ with 
\[
V_q= Ba_s\sum_{j=1}^{f} a_{jm+1}^{-1}= Ba_s\frac{D}{a_{fm+1}}, \qquad U_q=Ba_s\sum_{j=f+1}^{\infty} a_{jm+1}^{-1}.
\]
where
\[
D= a_{fm+1} \sum_{j=1}^{f} a_{jm+1}^{-1} \in\mathbb{Z}
\]
is an integer by \eqref{eq:divid}. 
By Proposition~\ref{uv} applied to
\[
u/v= \sum_{j=1}^{f-1} a_{mj+1}^{-1}, \quad M=M_{f-1}, \quad L=L_{f-1},
\] 
and an inductive argument,
we see that $(D,a_{mf+1})=1$.
Note hereby that in the base case $f=1$ of the induction, the hypothesis is easily checked.
On the other hand,
since $e_s\nmid B$ we have $B/e_s=Ba_s/a_{s+1}$ is not an integer.
Thus also $Ba_s/a_{mf+1}=(B/e_s)\cdot (a_{s+1}/a_{mf+1})\notin\mathbb{Z}$ since by assumption
of Case I we have $mf+1\geq s+1$ and thus $a_{s+1}|a_{mf+1}$ by
\eqref{eq:divid}.
Combining, we see that $V_q\notin \mathbb{Z}$.
Since the denominator in reduced form divides $a_{mf+1}$,
it has distance at least $a_{mf+1}^{-1}=d_f^{-1}$ from any integer.
Finally we can estimate
\[
|U_q| \leq Ba_s \cdot 2a_{m(f+1)+1}^{-1} = 2Qd_{f+1}^{-1},
\]
and the claim \eqref{eq:fre} and thus \eqref{eq:claim} follows.

Case II: Assume $s>mf$. Then by \eqref{eq:del} we may write $s=mf+h-1$ with an integer
$h\in\{ 2,3,\ldots,m\}$.
Case IIa: First assume $h\neq m$.
Write again $e_s= a_{s+1}/a_s\in\mathbb{Z}$ and $q=B\cdot a_{s}$ with some integer $B$ with $e_s\nmid B$ and $B=q/a_s\leq Q/a_s < a_{k+1}/a_s$.
Now observe that by assumption $h\neq m$ from \eqref{eq:tats}
we infer
\begin{equation} \label{eq:has}
e_s=d_f.
\end{equation}
Hence if we split $q\xi_h= Y_q+Z_q$ with 
\[
Y_q= Ba_s\sum_{j=nm+h,\; j\leq s } a_j^{-1}= B\sum_{j=nm+h,\; j\leq s } \frac{a_s}{a_j}\in\mathbb{Z}, \qquad Z_q=Ba_s\sum_{j=nm+h,\; j>s } a_j^{-1}
\]
then $Y_q\in\mathbb{Z}$ by \eqref{eq:divid}, and separating
the first term $Ba_s/a_{s+1}$ from the sum of $Z_q$
we may write 
\[
Z_q= q\xi_h-Y_q= \frac{Ba_s}{a_{s+1}}+ Ba_s\cdot O(a_{s+1+m}^{-1}) = \frac{B}{e_s} + O(Qd_{f+1}^{-1}).
\]
Now again since $e_s\nmid B$, 
the term $B/e_s$ is not an integer
and thus has distance at least
$1/e_s$ from any integer. So indeed by \eqref{eq:has} 
we see
\[
\Vert q\ux\Vert\geq \Vert q\xi_h\Vert= \Vert Z_q\Vert \geq e_s^{-1}-O(Qd_{f+1}^{-1}) = d_f^{-1}-O(Qd_{f+1}^{-1}).
\]

Case IIb: Finally assume $h=m$, or equivalently $s=k$. 
Then note that \eqref{eq:T} implies 
\[
e_s=\frac{a_{s+1}}{a_s}=\frac{a_{mf+m}}{a_{mf+m-1}}=L_f\leq a_{mf+1}= d_f,
\]
hence
$e_s\leq d_f$ again. Hence again the same argument 
as in Case IIa yields
\begin{equation}   \label{eq:JIPPY}
\Vert q\ux\Vert \geq \Vert q\xi_m\Vert \geq \frac{1}{e_s} -  O(Qd_{f+1}^{-1}) \geq \frac{1}{d_f} - O(Qd_{f+1}^{-1}),
\end{equation}
thus again \eqref{eq:claim} holds. The claim is proved in every case.

Now since $Q<a_{mf+m}=a_{mf+1}^m$ and $d_{f+1}=a_{mf+m}^{M_{f+1}}>a_{mf+1}^{M_{f+1}}$,
the remainder term in \eqref{eq:claim} can be bounded via
\begin{equation}  \label{eq:asinus}
Qd_{f+1}^{-1}\leq  a_{mf+m}d_{f+1}^{-1}=
a_{mf+1}^md_{f+1}^{-1}=a_{mf+1}^{m-M_{f+1}}.
\end{equation}
Choosing $M_{f+1}$ sufficiently large in the next step, 
this will be arbitrarily small.
On the other hand, by construction of $L_f$ for some 
$L_fa_{mf+m-1}\leq Q^{\prime}\leq (L_f+1)a_{mf+m-1}$ we have
\[
d_f^{-1} \geq  \Phi(Q^{\prime}),
\]
hence
\begin{equation} \label{eq:kkk}
d_f^{-1} \geq \Phi(Q^{\prime})=
\Phi(\alpha_f\cdot L_{f}a_{mf+m-1}), \qquad\quad 1\leq \alpha_f=
\frac{Q^{\prime}}{L_f a_{mf+m-1}}\leq \frac{L_f+1}{L_f}.
\end{equation}
In view of \eqref{eq:asinus} and \eqref{eq:claim}, again for given
$\epsilon_1>0$ a suitably large choice of $M_{f+1}$
in the next step will ensure that
\[
\Vert q\ux\Vert > (1-\epsilon_1)\Phi(Q^{\prime}),
\]
uniformly in $1\leq q\leq Q$.
Now by \eqref{eq:toin} we have $\alpha_f\to 1^{+}$ as $f\to\infty$,
hence for $\epsilon_2>0$ and $f\geq f_0(\epsilon_2)$ 
by property $(d3)$ we infer
\[
\Vert q\ux\Vert > (1-\epsilon_2)\cdot (1-\epsilon_1)\Phi(L_{f}a_{mf+m-1}).
\]
Since $L_fa_{mf+m-1}/Q=a_{m(f+1)}/Q>1$ is arbitrarily close to $1$ when $f$ is large enough by \eqref{eq:del}, for arbitrarily small $\epsilon_3>0$ again from $(d3)$
we infer
\[
\Vert q\ux\Vert > (1-\epsilon_2)(1-\epsilon_1) \cdot (1-\epsilon_3)\Phi(Q) =
(1-\epsilon_4)\Phi(Q), \quad \epsilon_4=(1-\epsilon_1)(1-\epsilon_2)(1-\epsilon_3).
\]
Since $q\leq Q$ was arbitrary,
this means $\psi_{\ux}(Q)>(1-\epsilon_4)Q$.
We may make $\epsilon_4$ arbitrarily small by choosing 
$f$ large enough and consequently $\epsilon_1, \epsilon_2, \epsilon_3$ small enough.
Now we may let the according $\epsilon_4=\epsilon_4(n)$ of the $n$-th step of the construction tend to $0$ 
as $n\to\infty$, and claim $(C2)$
follows for the induced $\ux$.

Since we can choose infinitely many distinct $M_n$ and thus
$a_{mn+1}$ in each step
of the construction in \S~\ref{ko},
our method gives rise to a continuum
of $\ux$ with the properties of the theorem. We prove the
stronger metrical assertion in~\S~\ref{metu} below.

\subsection{Proof of Corollary~\ref{rocky} }  \label{crproof}

Choose $\ux$ as in~\S~\ref{ko} for $\Phi(t)=ct^{-1/m}$, $c\in(0,1)$.
Let $\Omega:= \max_{1\leq i\leq m} |\underline{e}_i|$.
By relabelling indices if necessary,
we may assume $|\underline{e}_2|= \min_{1\leq i\leq m} |\underline{e}_i|=\chi$.
We show $\ux$
has Dirichlet constant $\Theta^{|.|}(\ux)=c\chi$ with respect to
$|.|$.
%
For simplicty write $\eta_i=\eta_i(q):= 
\Vert q\xi_i\Vert>0, 1\leq i\leq m$, so that 
$| \Vert q\ux\Vert |=|\sum \eta_i \underline{e}_i|$.
The proof of $(C1)$ 
in Case 1 with $Q=q=a_{mf}$ yields
again negligibly small values by
\[
| \eta_i \underline{e}_i |=\eta_i\cdot |\underline{e}_i|\leq \eta_i\cdot \Omega <q^{-1+\frac{1}{M_n}}\Omega=o(Q^{-1/m}),\qquad 1\leq i\leq m,
\]
independent of the norm. In Cases 2 and 3, for $Q=q=a_{mf+1}$ we see that \[
|\eta_2 \underline{e}_2|=(d_f^{-1}+o(d_f^{-1}))|\underline{e}_2|
=cQ^{-1/m}(1+o(1))\chi
\]  
whereas for $i\neq 2$ we get
\[
|\eta_i \underline{e}_i|\leq d_f^{-2}(1+o(1))|\underline{e}_i|\leq d_f^{-2}(1+o(1))\Omega=o(d_f^{-1})=o(Q^{-1/m}).
\]
Since $|\Vert q\ux\Vert| \leq \sum |\eta_i \underline{e}_i|$ by
triangle inequality,
we conclude in all cases $\Theta^{|.|}(\ux)\leq c\chi$. 
The reverse inequality $\Theta^{|.|}(\ux)\geq c\chi$ follows from
\[
|\Vert q\ux\Vert| \geq \max_{1\leq i\leq m} | \eta_i\underline{e}_i|\geq
\max_{1\leq i\leq m} \eta_i \min_{1\leq i\leq m} |\underline{e}_i|>c(1-\varepsilon)Q^{-1/m}\min_{1\leq i\leq m} |\underline{e}_i|= c(1-\varepsilon)\chi Q^{-1/m}
\]
for any $q<Q$,
where we used that $|.|$ is expanding and property $(C2)$.

\subsection{On relaxing conditions $(d2), (d3)$}   \label{nee}

We believe that for $m\geq 3$, the order in $(d2)$
can be significantly relaxed. Our proof above
followed the main outline from the
proof of~\cite[Theorem~2.5]{ichostrava} with
the special choice $\eta_1=\eta_2=\cdots=\eta_m=1/m$ (letter
$k$ was used in place of $m$ in~\cite{ichostrava}). 
Choosing $\eta_i$ differently, which essentially
means altering \eqref{eq:tats}, depending on a rough given
decay rate of $\Phi$, a similar approach may
ideally allow for replacing $(d2)$ by the weaker,
natural condition that $t\Phi(t)\to \infty$ (which coincides with
the exact condition $(d2)$ when $m=2$). However, some technical
obstacles have to be mastered.
Recall that for any 
$\ux\in \Rm\setminus \mathbb{Q}^m$ we have $\limsup_{Q\to\infty} Q\psi_{\ux}(Q)\geq 1/2$, as mentioned in \S~\ref{s1.1}.

Similarly, we believe that $(d3)$ can be dropped. Notice that we
have a free choice of $M_f$ in every step, so it suffices to
find some $M_f$ for which for given $\epsilon>0$ the induced $L_f$ satisfies $\Phi((L_f+1)a_{mf+1})>(1-\epsilon)\Phi(L_f a_{mf+1})$.
Now given $\epsilon>0$, for small enough $\delta=\delta(m,\epsilon)>1$ there are arbitrarily large $T$
so that $\Phi(\delta T)>(1-\epsilon)\Phi(T)$, otherwise it is easy to see that $(d2)$ cannot hold.
It remains however unclear to us if we can choose $T$ of the given form
$L_f a_{mf+1}$. In this matter it may be helpful that, as remarked in \S~\ref{ko}, we can relax \eqref{eq:h} by asking
$a_{mf+1}=M_n^{\prime} a_{mf}$ for large enough $M_n^{\prime}$.
%
%

\section{Proof of Theorem~\ref{F} for $m\geq 3$: metrical claim} \label{metu} 

\subsection{Special case $a_n=2^{c_n}$ }
\label{met}

Assume for simplicity first the $a_n$ constructed in~\S~\ref{ko}
are of the form $a_n=2^{c_n}$ for every $n$ with an increasing
sequence $c_n$, so that the binary expansion of the $\xi_i$ becomes
\[
\xi_i= \sum_{n=0}^{\infty} 2^{-c_{mn+i}}, \qquad 1\leq i\leq m.
\]
The proof is done in two steps.
The first key observation is that for given $\Phi$, we have some freedom in the construction of $\ux$ in~\S~\ref{ko}. We will find a Cantor type set consisting of $\uz\in\Rm$
sharing the same properties.
Then in the second step we use a very similar strategy as in the proof of~\cite[Theorem~2.1]{ichneu} based on a result of Tricot~\cite{tricot} to find the claimed lower bound for the packing
dimension of this set.

\underline{Step 1}: 
It is clear that
the binary digits of all $\xi_i$ as above at positions
$c_{mn}+1, c_{mn}+2, \ldots, c_{mn+1}-1$ are $0$ (whereas 
$\xi_m$ resp. $\xi_1$ has digit
$1$ at $c_{mn}$ resp. $c_{mn+1}$). Let small $\epsilon>0$ be given 
and $\gamma\in (0,1)$ as in the theorem.
Let $\mathcal{S}=\mathcal{S}(\gamma)\subseteq [0,1)^m$ be the set
of real vectors $\uz= (\zeta_1,\ldots,\zeta_n)$ whose coordinates have binary expansion $\zeta_i=\sum_{j\geq 1} g_{i,j}2^{-j}$
derived from the binary expansion of $\xi_i$, for $1\leq i\leq m$, by altering its digit from $0$ to an arbitrary digit 
$g_{i,j}\in\{0,1\}$ 
at places $j$ in the intervals
\[
I_n=\{ \lfloor c_{mn+1}(\gamma+\epsilon)\rfloor+1, \lfloor c_{mn+1}(\gamma+\epsilon)\rfloor+2,\ldots,
c_{mn+1}-1\}, \qquad n\geq 0.
\]
Note that still $g_{i,j}=0$ at places $j\in [c_{mn}+1, \lfloor c_{mn+1}(\gamma+\epsilon)\rfloor] \cap \mathbb{Z}$ for $1\leq i\leq m$, $n\geq 1$.

Then for any $\uz\in \mathcal{S}$,
the proof of $(C2)$ is analogous, upon some twists that we explain now. 
Since we include
certain terms $2^{-j}$ for $j \in\cup I_n$ to the partial sums
defining the $\zeta_i$, we need to slightly redefine
the integer $D=D(\uz)$ in Case I 
resp. $Y_q=Y_q(\uz)$ and $Z=Z_q(\uz)$ in Case II, 
depending on the choice of $\uz$. 
The coprimality condition in Case I is further 
guaranteed since at position $j=c_{mf+1}$
the digit of any $\zeta_1$ still equals $g_{1,c_{mf+1}}=1$
since $c_{mf+1}\notin \cup I_n$, so 
any $D(\uz)$ is odd but $a_{mf+1}=2^{c_{mf+1}}$
is a power of $2$. Similarly, the first, main term of $Z_q$ 
obtained from truncating the binary expansion of $\zeta_h$ after 
position $j=c_{s+1}$, is still
bounded from below by $d_f^{-1}$. Indeed, it may be written $qr_s=Ba_s\cdot r_s$ with the rational numbers 
$r_s= r_s(\zeta_h):= \sum_{j=c_s}^{c_{s+1}} g_{h,j}(\zeta_h)2^{-j}$ where $g_{h,j}(\zeta_h)$
is the binary digit of $\zeta_h$ at position $j$, depending
on the choice of $\uz$. But we have $g_{h,c_{s+1}}(\zeta_h)=g_{h,c_{mf+h}}(\zeta_h)=1$ for any $\uz\in \mathcal{S}$, since by $c_{mf+h}\notin \cup I_n$ we have not changed the digit of $\xi_h$ there.
So $r_s=v_s/2^{c_{s+1}}=v_s/a_{s+1}$ with some odd integer numerator $v_s$.
Finally, since $v_s$ is odd and $B\nmid (a_{s+1}/a_{s})$ as well by assumption, 
$qr_s=Bv_{s}a_s/a_{s+1}$ is not an integer and thus has distance at least $2^{-(c_{s+1}-c_s)}=a_s/a_{s+1}=d_f^{-1}$ from any integer. 

Moreover, we have chosen the left interval endpoints of the $I_n$ 
large enough that
we also satisfy $(C3)$ and Case 1 of $(C1)$ with the same
choice $q=a_{mn}=2^{c_{mn}}$, for any $\uz\in \mathcal{S}$. Indeed, the latter can be verified via
\[
\Vert q\uz\Vert \ll 2^{-(\gamma+\epsilon)c_{mn+1} + c_{mn} } < (2^{c_{mn+1}})^{-\gamma}= a_{mn+1}^{-\gamma}<Q^{-\gamma}< \Phi(Q), \quad Q<a_{mn+1},
\]
where we used $c_{mn+1}/c_{mn}\to \infty$
as $n\to\infty$, which we may assume
(since $c_{j}=\log_2 a_{j}$ this is a little stronger than
our original assumption $M_n^{\prime}= a_{mn+1}/a_{mn}\to \infty$),
and property $(d4(\gamma))$.
The proofs of the remaining
cases of $(C1)$ for any $\uz\in \mathcal{S}$ are again unchanged. Hence all $\uz\in \mathcal{S}$ satisfy
the claims of Theorem~\ref{F}.

\underline{Step 2}:
Now we show that $\mathcal{S}=\mathcal{S}(\gamma)$ 
has packing 
dimension at least $m(1-\gamma)$. This works similar as 
in the proof of~\cite[Theorem~2.1]{ichneu}. 
Let $\mu=\gamma^{-1}-1$.
First we claim that
we can write any given real vector $\underline{y}\in\Rm$ as the sum of an element of $\mathcal{S}$
and a vector with ordinary exponent of binary approximation 
(to be defined below) at least $\mu-\varepsilon$, for small $\varepsilon>0$ that
tends to $0$ as $\epsilon$ does.
More precisely, if we let
\[
\mathcal{V}_{m}^{(2)}(\lambda)=\{ \underline{x}\in\Rm: \liminf_{t\to\infty}\; (2^t)^{\lambda} \psi_{\underline{x}}(2^t)<\infty \}\subseteq 
\bigcup_{\tau \geq \lambda } \mathcal{W}_m(\tau), \qquad \lambda>0, 
\] 
then for some small $\varepsilon>0$ we claim
\begin{equation} \label{eq:OPE}
\mathcal{V}_{m}^{(2)}(\mu-\varepsilon)+ \mathcal{S}= \Rm.
\end{equation}
Given $\underline{y}\in\Rm$, we construct a representation $\underline{x}+\uz=\underline{y}$
for $\underline{x}\in \mathcal{V}_{m}^{(2)}(\mu-\varepsilon), \uz\in \mathcal{S}$.
We define $\uz\in\Rm$ as any
vector in $\mathcal{S}$ whose coordinates 
$\zeta_i$, $1\leq i\leq m$, have the same binary digit as the corresponding
component $y_i$ of $\underline{y}$ within the intervals $I_n$.
This is possible by the construction of $\mathcal{S}$. Then we let
$\underline{x}= \underline{y}-\uz$.
Consequently the coordinates of
$\underline{x}\in\Rm$ all have binary digits
$0$ in the entire intervals $I_n$.
Hence by construction of $I_n$, it is easy to check that
$\underline{x}\in\mathcal{V}_{m}^{(2)}(\mu-\varepsilon)$ for small enough $\varepsilon>0$. Indeed, by taking $q=2^{\lfloor c_{mn+1}(\gamma+\epsilon)\rfloor}$ and our choice of $\mu$, for small enough $\varepsilon>0$ depending on $\epsilon$ we have
\[
\Vert q\underline{x}\Vert\ll 2^{-|I_n|}\ll 2^{-(1-\gamma-\epsilon)c_{mn+1}}
\ll q^{-(\mu-\varepsilon)}.
\]
The construction is complete and
the claim is proved.

Write $\dim_H$ and $\dim_P$ for Hausdorff and packing dimension,
respectively. For simplicity set $\mathcal{V}= \mathcal{V}_{m}^{(2)}(\mu-\varepsilon)$.
Next we claim that
\begin{equation} \label{eq:RAT}
\dim_H(\mathcal{V})\leq \frac{m}{\mu-\varepsilon+1}= m\gamma+O(m\varepsilon),
\end{equation}
where the implied constant depends on $\gamma$ only.
This can be done by a standard covering argument and was
already observed in a more general form in~\cite[Lemma~5.6]{ichneu}.
See next paragraph for an alternative proof using Lemma~\ref{rem}.
Combining \eqref{eq:OPE}, \eqref{eq:RAT} with
a result by Tricot~\cite{tricot} and the well-known property
that the Hausdorff dimension of a set does not increase under a Lipschitz map, we conclude
\begin{align*}
\dim_P(\mathcal{S})\geq 
\dim_H(\mathcal{S}\times \mathcal{V})- \dim_H(\mathcal{V})
\geq \dim_H(\mathcal{S}+\mathcal{V})- \dim_H(\mathcal{V})
\geq m-m\gamma-O(m\varepsilon),
\end{align*}
and since $\epsilon$ and thus $\varepsilon$ can be arbitrarily small, 
the claim of the theorem.

\begin{remark}  \label{rehmar}
	In the light of the short proof of Corollary~\ref{ok} 
	in~\S~\ref{8} below,
	the above already implies the weaker claim that $Di_m(c)\setminus (Di_m(2^{-1/m}c) \cup Bad_m)\subseteq \boldsymbol{FS}_m$ has packing dimension at least $m-m\gamma$ for any $c\in(0,1]$. More generally, for $\Phi$
	as in Theorem~\ref{F}, analogous claims hold upon
	replacing the factor $(1-\varepsilon)$ in $(C2)$ by $2^{-1/m}$.
\end{remark}

\subsection{General case}  \label{tem}
Unfortunately, as indicated in Remark~\ref{rehmar},
we cannot guarantee that $a_n$ are not powers of $2$ 
in general, without losing information on the exact Dirichlet constant.
Here we explain how to alter the construction to the general case,
however omit a few technical details. Given any integer sequence $k_j$ 
satisfying \eqref{eq:divid}, i.e. $k_j|k_{j+1}$, it is not hard to verify that
any number in $[0,1)$ can be expressed
as $\sum_{j\geq 1} g_j/k_j$ with integers $g_j\in \{ 0,1,2,\ldots,k_{j+1}/k_j-1 \}$, via some kind of greedy expansion.
Call $k_j$ bases and $g_j$ digits.
If $k_{j+1}/k_j=b$ for $j\geq 1$ then this becomes 
just the usual $b$-ary expansion.
Now we choose $a_{mn+1}=2^{Z} a_{mn}$ for some integer $Z=Z_n$ so that
$M_n^{\prime}=2^{Z_n}$ in the notation of~\S~\ref{ko}, and apply
the above construction to the digits $g_{i,j}$ of $\xi_i$ 
in this setting, in place of the binary digits.
Concretely, we choose any $a_j$ as a base integer and call these main bases. At every main base integer $k_j=a_{u}$, for $i\equiv u \bmod m$
with $i\in \{1,2,\ldots,m\}$ 
we choose the digit of $\xi_i$ as $g_{i,j}=1$, and put $g_{i,j}=0$
for the other $i\not\equiv u\bmod m$, very similar to the binary digit construction above. We further define additional
intermediate bases within intervals $(a_{mn}, a_{mn+1})$ as follows.
Let $Y=Y_n<Z_n$ be an integer so that $2^{Y}a_{mn}=a_{mn+1}^{\gamma+\epsilon}$ with small
$\epsilon>0$, which is clearly possible since the approximation 
can be made precise up to a factor $2$. Now within
any interval $(a_{mn},a_{mn+1})$, we choose the intermediate bases
$2^{Y}a_{mn}, 2^{Y+1}a_{mn}, 2^{Y+2}a_{mn}, \ldots, 2^{Z-1}a_{mn}=a_{mn+1}/2$.
We define our sequence of bases $(k_j)_{j\geq 1}$
as the increasingly ordered union of all main and intermediate bases.
By a similar argument as in the binary construction~\S~\ref{met}, again we can choose 
the digits $g_{i,j}\in \{0,1\}$ of any component $\zeta_i$ with respect to the intermediate bases $2^{Y}a_{mn}, 2^{Y+1}a_{mn}, \ldots, 2^{Z-1}a_{mn}$ freely without violating
$(C1), (C2), (C3)$. Call $\mathcal{S}^{\ast}\subseteq \Rm$ the according set of real vectors $\uz$ induced by the above digit restrictions.
Since the good approximations are not powers of $2$, we need a slightly different argument than in~\S~\ref{met} for the optimal result.

\begin{lemma} \label{rem}
	Let $\boldsymbol{e}=(e_j)_{j\geq 1}$ be any strictly increasing sequence of positive integers (that thus tends to infinity) and $\tau>0$. Then the set 
	$\mathcal{V}_{m,\boldsymbol{e} }(\tau)$ of vectors $\underline{x}\in\Rm$ satisfying
	\begin{equation}  \label{eq:estili}
	\Vert e_j\underline{x}\Vert \leq e_j^{-\tau}, \qquad j\geq 1
	\end{equation}
	has Hausdorff dimension at most $m/(\tau+1)$.
\end{lemma}

\begin{proof}
The set $\mathcal{V}_{m,\boldsymbol{e} }(\tau)$ is clearly contained in the set of $\underline{x}$ for which estimate \eqref{eq:estili} has infinitely many solutions for any given subsequence $(e_{j_{k}})_{k\geq 1}$ of the $e_j$. By
the convergence case of Jarn\'ik-Besicovich Theorem~\cite{jarnik} 
in a setup involving approximations
functions (no monotonicity is required since $m>1$), choosing the approximation function $\Psi$ with support only on the $e_{j_{k}}$
and there equal to $\Psi(e_{j_{k}})=e_{j_{k}}^{-\tau}$,
the latter set has Hausdorff $\nu$-measure $0$ if the 
sum of $e_{j_{k}}^{m-\nu(\tau+1)}$ over $k\geq 1$ converges. As soon as $\nu>m/(\tau+1)$, by choosing a sparse enough subsequence $e_{j_{k}}$ of the $e_j$, the criterion obviously holds. This means the Hausdorff dimension the latter set is at most $m/(\tau+1)$, thus of
our original set $\mathcal{V}_{m,\boldsymbol{e} }(\tau)$ as well. 
\end{proof}

By a similar argument as in~\S~\ref{met}, 
we can write given $\underline{y}\in\Rm$ as a
sum $\underline{x}+ \uz$ where $\uz\in \mathcal{S}^{\ast}$ and $\underline{x}$ 
has the property that $\Vert q\underline{x}\Vert\ll q^{-(\mu-\varepsilon)}$ at the places $q=2^{Y_n}a_{mn}$ for 
$n\geq 1$. 
So we may apply Lemma~\ref{rem} to $e_j=2^{Y_j}a_{mj}$ and $\tau=\mu-\varepsilon$, which shows that the according set $\mathcal{V}_{m,\boldsymbol{e} }(\tau)$ 
of $\underline{x}$ has 
Hausdorff dimension at most $m/(\mu-\varepsilon+1)$ again.
The estimate 
for the packing dimension of $\mathcal{S}^{\ast}$
follows now analogously to~\S~\ref{met} from Tricot's result. 

\begin{remark}
It can be shown with aid of~\cite[Example~4.6, 4.7]{falconer}
that Lemma~\ref{rem} states the precise Hausdorff dimension
of $\mathcal{V}_{m,\boldsymbol{e} }(\tau)$, hence we cannot hope for an improvement by some refined treatment of this set.
The proofs of Theorems~\ref{hdd},~\ref{beides} below are based on this strategy.
\end{remark}

\section{Proof of Theorem~\ref{F} for $m=2$}  \label{ss}

For $m=2$, we have to alter our sequence $(a_n)_{n\geq 1}$.
Define $a_1$ and $a_2=a_1^2$, and for $n\geq 1$ recursively we set
\[
a_{2n+1}=a_{2n}^{M_n}, \qquad a_{2n+2}=\tilde{L}_na_{2n+1}
\]
where again $M_n$ is a fast growing sequence of integers, but now
\[
\tilde{L}_n= \{ \min z\in\mathbb{N}: z^{-1} < \Phi(Q): 1\leq Q\leq za_{2n+1} \}.
\]
By $(d2)$ this is a well-defined finite number.
Note that this is slightly different than $L_f$ in~\S~\ref{ko}.
Indeed,
now the reverse inequality $\tilde{L}_n>d_n:=a_{2n+1}$ 
holds by $(d1)$ for all $n\geq 1$, and 
clearly $\tilde{L}_n\to\infty$ follows. Then again define
$\xi_i$ as in \eqref{eq:sinus}.

The proof of $(C3)$ is identical to \S~\ref{ko} by considering
$Q=q=a_{2n}$ for large $n$. For $(C1)$
we again let $Q$ be large and $k$ the index
with $a_k\leq Q<a_{k+1}$. We consider the same cases again.
Case 1 is very similarly inferred from $(d2)$, and again 
leads to $\Vert q\ux\Vert= \Vert q\xi_1\Vert <  \Phi(Q)$. 
Case 2 of $(C1)$ is empty.
In Case 3 of $(C1)$, we now instead 
of \eqref{eq:et} get a bound
\begin{align}  \label{eq:EE}
\Vert q\ux\Vert= \Vert q\xi_2\Vert &\leq \frac{1}{ \tilde{L}_f }+ O(Qd_{f+1}^{-1}).
\end{align}
By definition of $\tilde{L}_f$ and since we can choose $M_{f+1}$ in the next step arbitrarily large, we again see $\Vert q\ux\Vert= \Vert q\xi_2\Vert < \Phi(Q)$ for any $Q<a_{2n+2}$.

We follow the proof of $(C2)$ as in \S~\ref{proof}.
Again we let $Q=a_{2f+2}-1$ for large $f$ and consider the 
same cases I, II.
In Case I we get the same estimate \eqref{eq:fre} by the same argument.
However, since now $\tilde{L}_f>d_f$ we have $d_f^{-1}>\tilde{L}_f^{-1}$.
In case IIa we again get a lower bound 
$\Vert q\ux\Vert \geq e_s^{-1}-O(Qd_{f+1}^{-1})$ but now $e_s=\tilde{L}_f$. Case IIb also gives the same bound $e_s^{-1}-O(Qd_{f+1}^{-1})$.
Since we noticed $1/\tilde{L}_f<1/d_f$, in any case we get
\begin{equation} \label{eq:FF}
\Vert q\ux\Vert \geq \frac{1}{ \tilde{L}_{f}}+ O(Qd_{f+1}^{-1}).
\end{equation}
The error term can be made $o(\tilde{L}_f^{-1})$
if we choose $M_{f+1}$ large enough in every step. Moreover,
$(\tilde{L}_f-1)^{-1}\geq \Phi(Q^{\prime})$ for some
$Q^{\prime}\in [ (\tilde{L}_f-1)a_{2f+1}, a_{2f+2})$ by definition of $\tilde{L}_f$, hence
\[
\Vert q\ux\Vert \geq (1-\epsilon_1)\Phi(Q^{\prime})=(1-\epsilon_1)\Phi(\tilde{\alpha}_f \tilde{L}_fa_{2f+1}) , \qquad 
\tilde{\alpha}_f=\frac{ Q^{\prime} }{ \tilde{L}_fa_{2f+1 }}\geq \frac{\tilde{L}_f-1 }{ \tilde{L}_f }.
\]
Now $\tilde{\alpha}_f\to 1^{-}$, so
we can use the latter condition in $(d3)$ to conclude very similarly as in \S~\ref{proof} (note that we again require the condition on the right-sided lower limit $\alpha\to1^{+}$ as well
for the final step of the argument, or some similar property). 
The metrical claim also follows analogously to~\S~\ref{met}.
We omit the details.

\begin{remark}  \label{hr}
	For general $m\geq 2$, defining again $a_{mn+1}=a_{mn}^{M_n}$
	and constant quotients $a_{mn+j+1}/a_{mn+j}=\tilde{L}_n$ for $1\leq j\leq m-1$ with
	\[
	\tilde{L}_n= \{ \min z\in\mathbb{N}: z^{-1} < \Phi(Q): 1\leq Q\leq z^{m-1}a_{2n+1} \},
	\]
	analogous arguments would lead to an alternative, slightly shorter proof of
	Theorem~\ref{F} upon assuming the latter condition in $(d3)$.  
	We preferred to include the longer proof for $m\geq 3$ 
	since when adapting the alternative above construction to Cantor sets
	as in~\S~\ref{fra}, without additional argument the 
	bound in Theorem~\ref{2} would become weaker.
\end{remark}

\section{Proof of Hausdorff dimension estimates}

\subsection{Metric preliminaries}  \label{putz}

To prove Theorem~\ref{hdd} resp. Theorem~\ref{beides},
similar as in~\S~\ref{metu}, 
in short, we construct a Cantor type subset of the 
corresponding set $\textbf{F}_{m,c}$ resp. $\textbf{F}_{m,c}\cap \mathcal{W}_m(\lambda)$ whose Hausdorff dimension can be estimated/evaluated.
The fractal set will be as in the following
lemma, for optimized parameters $\gamma_1, \gamma_2$ under certain
side conditions. The notation $A\asymp B$ means $A\ll B\ll A$ in the sequel.

\begin{lemma}  \label{ohlele}
	Let $m\geq 2$ an integer and $\gamma_2\geq \gamma_1>1$ be real numbers. Assume $(c_n)_{n\geq 1}$ 
	is an increasing sequence of integers and let $h_n=c_{mn}$
	and $H_n=2^{h_n}$ for $n\geq 1$. Assume
	\begin{equation}  \label{eq:SG}
	H_{n}\asymp H_{n-1}^{\gamma_2 m}, \qquad n\geq 2.
	\end{equation}
	Let a sequence $(\delta_n)_{n\geq 1}$ satisfy
	$\delta_n>\gamma_1$ for $n\geq 1$ and $\delta_n= \gamma_2+o(1)$ 
	as $n\to\infty$. For $1\leq i\leq m$, let
	$\mathcal{Q}_i\subseteq [0,1)$ be the set of real numbers $\zeta_i$ 
	whose binary expansion $\zeta_i= \sum_{j\geq 1} g_{i,j}2^{-j}$
	has an arbitrary digit $g_{i,j}\in\{ 0,1 \}$ at places of the form 
	\begin{enumerate}
		\item[(i)] For $1\leq i\leq m$ in intervals $j\in [\gamma_1 h_{n},\delta_n h_{n}-1]\cap \mathbb{Z}$
		\vspace{0.1cm}
		\item[(ii)] For $3\leq i\leq m-1$, in intervals $j\in [2\delta_n h_{n}+1,i\delta_n h_{n}-1]\cap \mathbb{Z}$
		\vspace{0.1cm}
		\item[(iii)] For $i=m$, in intervals
		$j\in [2\delta_n h_n+1, h_{n+1}-1]\cap \mathbb{Z}$ 
	\end{enumerate}
	for all $n\geq 1$,
	and a prescribed digit $g_{i,j}\in\{ 0,1 \}$ elsewhere.
	Then their Cartesian
	product $\prod \mathcal{Q}_i$ has Hausdorff dimension at least 
	\begin{equation}  \label{eq:01}
	\dim_H(\prod_{i=1}^{m} \mathcal{Q}_i) \geq 2\frac{ \gamma_2-\gamma_1 }{\gamma_1(\gamma_2m-1) } + \sum_{i=3}^{m} \min\left\{ \frac{ m(\gamma_2-\gamma_1)+i-2 }{ 2(m\gamma_2-1)
	} , \frac{ (i-1)\gamma_2 - \gamma_1 }{  \gamma_1(m\gamma_2-1) }   \right\}, 
	\end{equation}
	and alternatively
	\begin{equation} \label{eq:02} 
	\dim_H(\prod_{i=1}^{m} \mathcal{Q}_i)\geq m\cdot \frac{ \gamma_2-\gamma_1 }{\gamma_1(\gamma_2m-1) }.
	\end{equation}
\end{lemma}

We prove Lemma~\ref{ohlele}.
Our sets $\mathcal{Q}_i$ can be interpreted as a special case
of a construction from Falconer's book~\cite[Example~4.6]{falconer}.

\begin{proposition}[Falconer]  \label{falke}
	Let $[0,1] = E_0 \supseteq E_1 \supseteq E_2 \supseteq \cdots$
	be a decreasing sequence of sets, with each $E_n$ a union of a finite number of
	disjoint closed intervals (called $n$-th level basic intervals), with each interval of
	$E_{n-1}$ containing $P_{n}\geq 2$ intervals of $E_{n}$, 
	which are separated by gaps of length at least $\epsilon_n$,
	with $0<\epsilon_{n+1}<\epsilon_{n}$ for each $n$,
	which tend to $0$ as $n\to\infty$. Then the set
	\[
	F= \bigcap_{i\geq 1} E_i
	\]
	satisfies 
	\[
	\dim_H(F)\geq \liminf_{n\to\infty} \frac{ \log(P_1P_2 \ldots P_{n-1})}{-\log (P_n \epsilon_n)}.
	\]
\end{proposition}

	From our setup it can be seen that $F=\mathcal{Q}_i$ for 
	$1\leq i\le m$ from Lemma~\ref{ohlele} meet the requirements
of Proposition~\ref{falke} with parameters
\begin{equation}  \label{eq:GS0}
P_{n}\asymp H_{n}^{\delta_n - \gamma_1} = H_n^{\gamma_2-\gamma_1-o(1)}, \qquad
\epsilon_{n}\asymp H_n^{-\delta_n}= H_n^{-\gamma_2+o(1)}
\end{equation}
as $n\to\infty$ for all $1\leq i\leq m$, and alternatively with
\begin{equation}  \label{eq:GS}
P_{2n}\asymp H_n^{\gamma_2-\gamma_1+o(1)}, \;
\epsilon_{2n}\asymp H_n^{-\gamma_2+o(1)}, \; P_{2n+1}\asymp H_n^{(i-2)\gamma_2+o(1)}, \;
\epsilon_{2n+1}\asymp H_n^{-i\gamma_2+o(1)}, 
\end{equation}
for $3\leq i\leq m$.
We provide more details. 
Assume the interval construction is done up to level
$2n-1$ which prescribes digits up to position $\lceil \gamma_1h_n\rceil-1$.
Then the free binary digit choice within 
$j\in [\gamma_1 h_n, \delta_n h_n)=[\gamma_1 h_n, (\gamma_2+o(1))h_n)$
means that we split each interval given after
step $2n-1$ in the next step into $2^{\delta_n h_n-1- \lfloor \gamma_1 h_n\rfloor }= 2^{(\gamma_2-\gamma_1+o(1))h_n}=H_n^{\gamma_2-\gamma_1+o(1)}$ subintervals,
each of length $\asymp 2^{-\gamma_1 h_{n+1}}=H_{n+1}^{-\gamma_1}$ due to the subsequent digits
vanishing until position $j=\lfloor \gamma_1 h_{n+1}\rfloor$, and two neighboring intervals roughly at distance 
$\epsilon_{2n}\asymp 2^{-\delta_n h_n}-2^{-\gamma_1 h_{n+1}}= H_{n}^{-\gamma_2+o(1)}-H_{n+1}^{-\gamma_1}= H_n^{-\gamma_2+o(1)}$ 
apart. A very similar idea applies in the next step to estimate
$P_{2n+1}, \epsilon_{2n+1}$, where we distinguish between 
various $i$ and for $i>2$ use that for $i=m$ we also 
have $H_{n+1}= H_n^{i\gamma_{2}(1+o(1))}$ by \eqref{eq:SG}.
Inserting \eqref{eq:GS}, \eqref{eq:SG} in Proposition~\ref{falke}
we may omit lower order terms and
obtain for $3\leq i\leq m$ that
\begin{align*}
\dim_H(\mathcal{Q}_i)&\geq \liminf_{n\to\infty} \frac{ \log(P_1P_2 \ldots P_{n-1})}{-\log (P_n \epsilon_n)}\\
&=\min \left\{ \liminf_{n\to\infty} 
\frac{ \log(P_1P_2 \ldots P_{2n})}{-\log (P_{2n+1} \epsilon_{2n+1})},
\liminf_{n\to\infty} \frac{ \log(P_1P_2 \ldots P_{2n-1})}
{-\log (P_{2n}\epsilon_{2n}) }
\right\}  \\
&= \min\left\{  \frac{ \left(\gamma_2-\gamma_1 + \frac{ (i-2)\gamma_2}{\gamma_2 m }\right) \sum_{j=0}^{\infty} (\gamma_2 m)^{-j} \log H_n } { 2\gamma_2\log H_n } , 
\frac{ \frac{(i-2)\gamma_2 + \gamma_2-\gamma_1}{\gamma_2 m} \sum_{j=0}^{\infty} (\gamma_2 m)^{-j}\log H_n }{ \gamma_1\log H_n }  \right\} \\
&= \min\left\{ \frac{ m(\gamma_2-\gamma_1)+i-2 }{ 2(\gamma_2 m-1)
} , \frac{ (i-1)\gamma_2 - \gamma_1 }{  \gamma_1(\gamma_2 m-1) }   \right\}, 
\end{align*}
where the last identity requires short computations
involving the geomtric sum formula. Similarly, for any $1\leq i\leq m$
from \eqref{eq:GS0}, \eqref{eq:SG} and Proposition~\ref{falke} we get
\[
\dim_H(\mathcal{Q}_i)\geq 
\frac{ (\gamma_2-\gamma_1)\sum_{j=1}^{\infty} (\gamma_2m)^{-j} \log H_n }{\gamma_1 \log H_n } = 
\frac{(\gamma_2-\gamma_1)\frac{1}{1-\frac{1}{\gamma_2m}}}
{m\gamma_1\gamma_2}= \frac{ \gamma_2-\gamma_1 }{\gamma_1(\gamma_2m-1) }.
\] 
Combining the respective estimates with the general fact $\dim_H(\prod A_i)\geq \sum \dim_H(A_i)$ for any $A_1, \ldots,A_m\subseteq \Rm$, see~\cite{falconer},
with $A_i=\mathcal{Q}_i$ proves the claims of the lemma.

\subsection{Proof of Theorem~\ref{hdd}}  \label{put}

We assume $m\geq 3$ here, for $m=2$ the proof works very similarly.
We show how to derive Theorem~\ref{hdd} from Lemma~\ref{ohlele}.
Assume the real parameters $\gamma_1, \gamma_2$ satisfy the stronger hypothesis
\begin{equation}  \label{eq:dno}
m(\gamma_1-1) > \gamma_2\geq \gamma_1 > 1+\frac{1}{m}.
\end{equation}
In fact the setup \eqref{eq:dno} automatically 
requires $\gamma_1>1+\frac{1}{m-1}$.
Take $(a_{n})_{n\geq 1}$ the sequence constructed in~\S~\ref{ko}
with the specialization $\Phi(t)=ct^{-1/m}$ and 
assume for the moment $a_n=2^{c_n}$ for integers $c_n$.
Take $M_n^{\prime}\asymp a_{mn}^{\gamma_2-1}$ the integer power of $2$ 
closest to $a_{mn}^{\gamma_2-1}$ for $n\geq 1$, so that we have
\begin{equation}  \label{eq:TOR}
a_{mn+1}= M_n^{\prime}a_{mn} \asymp a_{mn}^{\gamma_2}=
2^{c_{mn}\gamma_2}, \qquad n\geq 1,
\end{equation}
and let for $n\geq 1$ further
\[
H_n=a_{mn}, \qquad\quad h_n=c_{mn}. 
\]
Note that our particular case $\Phi(t)=ct^{-1/m}$ implies \eqref{eq:SG}.
Indeed then $L_n\asymp a_{mn+1}$ and $H_{n+1}=a_{m(n+1)}= L_na_{mn+1}^{m-1}\asymp a_{mn+1}^{m}\asymp a_{mn}^{\gamma_2m}=H_n^{\gamma_2m}$.
(For general $\Phi$ under $(d4(\gamma))$ we would get $H_{n-1}^{\gamma_2 m}\gg H_n\gg H_{n-1}^{\gamma_2/\gamma}$.)
Let $\delta_n>0$ be defined by
\[
a_{mn+1}= a_{mn}^{\delta_n}, \qquad n\geq 1,
\]
which satisfies $\delta_n= \gamma_2+o(1)$ as $n\to\infty$ by \eqref{eq:TOR}, more precisely $H_n^{\delta_n}\asymp H_n^{\gamma_2}$.
(The definition of $\delta_n$ agrees with $M_n$ in~\S~\ref{ko}, however
since we do not assume it is an integer here, we prefer to change notation
for clarity.)

Consider the real numbers $\xi_i$ as in~\S~\ref{ko} when $a_{n}=2^{c_n}$ as above, i.e. $\xi_i=\sum_{n\geq 0} 2^{-c_{mn+i}}$. Then 
in the binary expansion the digit of $\xi_i$ is $1$ at 
places $c_{mn+i}, n\geq 0$, and $0$ elsewhere.
Then we consider the sets 
$\mathcal{Q}_i= \mathcal{Q}_i(\gamma_1, \gamma_2)\subseteq \mathbb{R}$, $1\leq i\leq m$, 
as in Lemma~\ref{ohlele} consisting
of the real numbers $\zeta_i=\sum g_{i,j}2^{-j}$, where
$g_{i,j}\in \{0,1\}$ depends on $\zeta_i$, obtained 
from the $\xi_i$ when we 
change the binary digit $g_{i,j}$ of $\xi_i$ from $0$ to an arbitrary digit
in $\{0,1\}$ in the intervals of type (i), (ii), (iii).
Then, any $\zeta_i= \sum_{j\geq 1} g_{i,j}2^{-j}\in \mathcal{Q}_i$ still has binary digit $g_{i,j}=0$ at places $j$
within the following intervals for all $n\geq 1$:
\begin{enumerate}
	\item[$(i^{\ast})$]  For $1\leq i\le m-1$, within 
	intervals 
	\[
	j\in [\max\{i,2\}\delta_{n-1} c_{m(n-1)}+1, \gamma_1 c_{mn})\cap \mathbb{Z}=[\max\{i,2\}\delta_{n-1} h_{n-1}+1, \gamma_1 h_{n})\cap \mathbb{Z}, 
	\]
	which contains $[h_{n}+1, \gamma_1 h_{n})\cap \mathbb{Z}$.
	\vspace{0.1cm}
	\item[$(ii^{\ast})$] For $i=m$, within 
	intervals
	\[j\in [c_{mn}+1,\gamma_1 c_{mn})\cap \mathbb{Z}=[h_n+1, \gamma_1 h_n)\cap \mathbb{Z}.
	\]
	\item[$(iii^{\ast})$]  For $i=2$, within 
	intervals 
	\[
	j\in [c_{mn+1}+1, 2c_{mn+1}-1]\cap \mathbb{Z}=[\delta_n h_{n}+1, 2\delta_n h_{n}-1]\cap \mathbb{Z}.
	\]
	\item[$(iv^{\ast})$]  For $i=1$ and $3\leq i\le m$, within 
	intervals 
	\[
	j\in [c_{mn+1}+1, 2c_{mn+1}]\cap \mathbb{Z}=[\delta_n h_{n}+1, 2\delta_n h_{n}]\cap \mathbb{Z}.
	\]
\end{enumerate}

Note that $i\delta_n h_n=ic_{mn+1}=c_{mn+i}$ for $1\leq i\leq m-1$ 
and $n\geq 1$ by \eqref{eq:tats}.
We claim 

\begin{lemma}  \label{mure}
	Any $\uz\in\prod\mathcal{Q}_i$ as above
	satisfies 
	\[
	\limsup_{Q\to\infty} Q^{1/m}\psi_{\uz}(Q)=c,
	\] 
	i.e. $(C1^{\prime}), (C2)$ for $\Phi(t)=ct^{-1/m}$,
	where $(C1^{\prime})$ is $(C1)$ up to a admitting a factor
	$1+\varepsilon$ in the right hand side.
\end{lemma}

\begin{proof}[Proof of Lemma~\ref{mure}]
	Take $\uz\in\prod\mathcal{Q}_i$. 
	We first verify $(C1^{\prime})$. 
	Assume we are in Cases 2,~3 of the proof of $(C1)$ in~\S~\ref{proof}.
	We again consider integers $q=a_{mf+1}=2^{c_{mf+1}}= H_f^{\delta_f} \asymp H_f^{\gamma_2}$ that satisfy $q<Q<a_{m(f+1)}\leq a_{mf+1}^{m}$.
	By $(iii^{\ast})$,
	we have essentially the same estimates for $\Vert q\zeta_2\Vert$ as in~\S~\ref{proof} 
	and by $(iv^{\ast})$ the other $\Vert q\zeta_i\Vert$, $i\neq 2$, take smaller values, so
	\[
	\Vert q\uz\Vert= \Vert q\zeta_2\Vert 
	\leq a_{mf+1}^{-1} + O(a_{mf+1}a_{m(f+1)+1}^{-1}),\qquad \uz\in \prod \mathcal{Q}_i. 
	\]
	As in the proof in~\S~\ref{proof}, by construction $a_{mf+1}^{-1}=d_f^{-1}<\Phi(Q)=cQ^{-1/m}$ for any $Q< a_{m(f+1)}$.
	Hence, 
	upon admitting a factor $1+\varepsilon$ coming from the
	lower order error
	term (which however we cannot control as freely here in view
	of \eqref{eq:TOR}), we satisfy in $(C1)$.
	In Case 1 of the proof of $(C1)$ in~\S~\ref{proof}, we use $(i^{\ast}), (ii^{\ast})$ and
	$\gamma_1-1 > \gamma_2/m$ from \eqref{eq:dno} when
	$\Phi(t)=ct^{-1/m}$ (under $(d4(\gamma))$ in general $\gamma_1-1> \gamma_2\gamma$)
	to guarantee
	the same estimate. Indeed, for $q=a_{mf}=2^{h_f}=H_f$
	and any $Q<a_{mf+1}=H_f^{\delta_n}\leq H_f^{\gamma_2+o(1)}$ we calculate
	\[
	\Vert q\uz\Vert \ll 2^{-\gamma_1 h_f+ h_f} = H_f^{-(\gamma_1-1)} 
	\ll a_{mf+1}^{-\frac{ \gamma_1-1 }{ \gamma_2 }+o(1)}
	= o(a_{mf+1}^{-1/m}), \qquad \text{as}\; f\to\infty,\quad \uz\in \prod \mathcal{Q}_i.
	\]
	Hence $\Vert q\uz\Vert< cQ^{-1/m}=\Phi(Q)$ for $f\geq f_0$.
	The proof of $(C2)$ is almost analogous to
	the classical case in~\S~\ref{proof}, with two minor changes.
	Firstly, we again use $\gamma_1 > 1+1/m$ 
	from \eqref{eq:dno} and $(i^{\ast}), (ii^{\ast})$ for
	the error term $O(Qd_{f+1}^{-1})=O(a_{m(f+1)}a_{m(f+1)+1}^{-1})$ to be negligible.
	Indeed, by $\gamma_1>1+1/m$ from \eqref{eq:dno} and $(i^{\ast}), (ii^{\ast})$,
	the error term can be estimated $\ll a_{m(f+1)}^{1-\gamma_1}=o(a_{m(f+1)}^{-1/m})$, 
	while by construction the main term
	is just slightly smaller than $\Phi(a_{m(f+1)})=ca_{m(f+1)}^{-1/m}$.
	(In general
	under condition $(d4(\gamma))$ on $\Phi$, we require $\gamma_1>1+\gamma$
	for the same conclusion.)  
	Secondly
	our digital variations from (i) resp. (ii) induce slightly
	different integers $D=D(\uz)$ in Case I resp. $Y_q=Y_q(\uz)$ 
	and $Z_q=Z_q(\uz)$ in Case II, now depending
	on the choice of $\uz\in \prod \mathcal{Q}_i$. 
	The according crucial properties hold again for similar
	reasons as in~\S~\ref{met}.
	Note that the remainder terms are essentially 
	unaffected in view of $(i^{\ast}), (ii^{\ast})$.
	Hence our claim is proved. 
\end{proof}

By Lemma~\ref{mure}, and since $\lambda(\uz)\geq 1>1/m$ for any $\uz\in \prod \mathcal{Q}_i$
is easy to see by choosing integers $q=a_{mn+1}=\delta_n h_n$
in view of $(iii^{\ast}), (iv^{\ast})$, see also~\S~\ref{pbeides} below,
the set $\textbf{F}_{m,c}$ contains $\prod \mathcal{Q}_i$. Thus
we may apply Lemma~\ref{ohlele} to bound its Hausdorff dimension 
from below as in \eqref{eq:01}.
For the asymptotical bound as $m\to\infty$,
note that for any pair $\gamma_1, \gamma_2$ with
$\gamma_1>\gamma_2/m+1$ and
$\gamma_2>1+\frac{1}{m-1}$ the assumption \eqref{eq:dno} holds.
The bound \eqref{eq:01} clearly decays in $\gamma_1$. 
So basically we want to maximize \eqref{eq:01} 
over $\gamma_1, \gamma_2$ satisfying $\gamma_1=\gamma_2/m+1$ and $\gamma_2>1+\frac{1}{m-1}$. To give an asymptotical estimate,
we choose $\gamma_1$ just slightly larger than $\gamma_2/m+1$ and
let $\gamma_2=m^{\beta}$ for fixed $\beta\in(0,1)$. Then $\gamma_1=1+ O(m^{\beta-1})$ is just slightly larger than $1$, so it is 
asymptotically negligible in \eqref{eq:01}. 
We may further omit the small
positive first expression $2(\gamma_2-\gamma_1)/(\gamma_1(m\gamma_2-1))$.
For large $m$ then
we check that the left expression in the minimum of \eqref{eq:01}
is of order $1/2+o(1)$
while the right is of order $i/m+o(m)$.
Hence the minimum in \eqref{eq:01}
equals roughly $i/m$ for $3\leq i\leq m/2+o(m)$ and
$1/2+o(1)$ for $m/2+o(m)\leq i\leq m$. Thus as $m\to\infty$ we bound 
the sum in \eqref{eq:01} from below by
\[
\left(\sum_{i=3}^{m/2 } \frac{i}{m} - o(m) \right) + \left( \sum_{i=m/2 }^{m} \frac{1}{2} - o(m)\right)= \frac{m }{8} + \frac{m }{4} -o(m)= \frac{3}{8}m-o(m),
\]
the claimed asymptotical estimate \eqref{eq:cosinus2}.

Alternatively
by Lemma~\ref{ohlele} and Lemma~\ref{mure} the Hausdorff dimension of $\textbf{F}_{m,c}$ 
can be estimated from below by \eqref{eq:02}. 
We now prove \eqref{eq:sinus2} by optimizing 
the parameters $\gamma_1, \gamma_2$.
We already noticed that the expression
in \eqref{eq:02} decreases in $\gamma_1$. Hence 
in view of \eqref{eq:dno} we again take parameters related by the identity 
$\gamma_1= \frac{\gamma_2}{m}+1 + \epsilon$
with small $\epsilon>0$ and bare in mind that we require $\gamma_2 > 1+\frac{1}{m-1}$ 
for the conditions
\eqref{eq:dno} to be satisfied. Since $\epsilon>0$ can be arbitrarily small, we want to maximize the function
\[
\gamma_2\longmapsto m\frac{ \frac{m-1}{m}\gamma_2 - 1}{ (\frac{\gamma_{2}}{m}+1)(m\gamma_2-1) } = 
\frac{ (m-1)\gamma_2 - m  }{ \gamma_2^2 + (m-\frac{1}{m})\gamma_2-1  }
\]
over $\gamma_2>1+\frac{1}{m-1}$. By differentiation
we verify that the maximum is attained at
\[
\gamma_2= \frac{m+\sqrt{ m(m^2-m+1) }}{m-1}>1+\frac{1}{m-1}
\]
the positive solution of $(m-1)x^2-2mx+m(1-m)=0$.
Inserting in the function gives the desired estimate \eqref{eq:sinus2}
after a short simplification.

Finally,
we may generalize Lemma~\ref{ohlele} to the situation where $H_n$
are not powers of $2$, and consequently
drop the assumption $a_n=2^{c_n}$, essentially
by the argument explained in detail in~\S~\ref{tem}. We omit 
recalling the strategy.

\subsection{Proof of Theorem~\ref{beides} }  \label{pbeides}

Assume $\gamma_1, \gamma_2$ satisfy \eqref{eq:dno} and also
\begin{equation}  \label{eq:leilei}
(\gamma_1-1)^2> \gamma_2.
\end{equation}
We again assume for simplicity that $a_n=2^{c_n}$ for integers $c_n$,
the general case can be obtained as in \S~\ref{tem}.
Recall the sets $\mathcal{Q} _i$ induced by $\gamma_1, \gamma_2$, constructed in~\S~\ref{put}
from freely altering binary digits of $\ux$ from~\S~\ref{ko}
in intervals of type (i), (ii), (iii).
We slightly alter $\mathcal{Q}_1$ by
imposing the addititonal condition (iv)
that any $\zeta_1\in \mathcal{Q}_1$ has binary digit $g_{1,j}=1$
at position $j=\lfloor \gamma_1 h_n\rfloor$, for $n\geq 1$.
Denote this set by $\mathcal{Q}_1^{\ast}\subseteq \mathcal{Q}_1$.
We claim 
\begin{lemma} \label{mitp}
	For any $\uz\in \mathcal{Q}_1^{\ast}\times \prod_{i=2}^{m} \mathcal{Q}_i$ as above we have
	\[
	\lambda(\uz)= \gamma_1-1.
	\]	
\end{lemma}

We believe that in fact $\lambda(\uz)= \max\{ \gamma_1-1,1\}$
for generic $\uz\in \prod_{i=1}^{m} \mathcal{Q}_i$ 
whenever $\gamma_1, \gamma_2$ are related by \eqref{eq:dno},
however we are unable to prove it. The lower bound $1$ hereby comes
from $(iii^{\ast}), (iv^{\ast})$.
The proof of Lemma~\ref{mitp} relies on the following standard result
on rational approximation to a single real number.

\begin{proposition} \label{kp}
	Let $x\in\mathbb{R}$. Assume for a reduced fraction $p/q\in \mathbb{Q}$ and $\tau>2$ we have
	\[
	|x- \frac{p}{q}| = q^{-\tau}.
	\]
	Then for any rational $\tilde{p}/\tilde{q}\neq p/q$ 
	with $q\leq \tilde{q}\ll q^{\tau-1}$ for a sufficiently small absolute implied constant, we have $|x-\tilde{p}/\tilde{q}|\geq \tilde{q}^{-2}/2$ or equivalently $\Vert \tilde{q}x\Vert \geq \tilde{q}^{-1}/2$.
\end{proposition}

Proposition~\ref{kp} follows from Legendre Theorem stating 
that $|r/s-x|<s^{-2}/2$ implies that $r/s$ (after reduction)
must be a convergent of the continued fraction
expansion of $x$, and the relation
$s_{k+1} \asymp |s_{k}x-r_{k}|^{-1}$ between two consecutive convergents $r_{k}/s_k$ and $r_{k+1}/s_{k+1}$. See~\cite[Proposition~4.2]{ichdimproduct} for a short proof 
of the latter fact. Alternatively, Minkowski's Second Convex Body
Theorem directly implies Proposition~\ref{kp}, see also~\cite{khint}.

\begin{proof}[Proof of Lemma~\ref{mitp}]
	Let $\uz\in \mathcal{Q}_1^{\ast}\times \prod_{i=2}^{m} \mathcal{Q}_i$ be given.
	By properties $(i^{\ast}), (ii^{\ast})$ from~\S~\ref{put}, the integers $H_n=2^{h_n}=a_{mn}$ induce an estimate
	\begin{equation}  \label{eq:HG}
	\Vert H_{n}\uz\Vert \asymp \Vert H_n\xi_1\Vert\asymp
	2^{-h_n \gamma_1+ h_n} = H_n^{-(\gamma_1-1)}, \qquad\qquad n\geq 1.
	\end{equation}
	For the lower bound we have used the non-zero digit 
	assumption (iii). 
	The lower estimate $\lambda(\uz)\geq \gamma_1-1$ follows
	directly from \eqref{eq:HG}.  
	Assume now conversely to the claim of the lemma
	that we have strict inequality $\lambda(\uz)>\gamma_1-1$.
	Then for some $\varepsilon>0$, there are arbitrarily large integers $q>0$ with the property
	\begin{equation}  \label{eq:tauid}
	\Vert q\uz\Vert < q^{-(\gamma_1-1)-\varepsilon}.
	\end{equation}
	Let $q$ be such an integer and let $f$ be the index with 
	$H_f\leq q< H_{f+1}$. 
	Recall that $H_f=a_{mf}$ and $H_{f+1}=a_{m(f+1)}$
	and the notation $d_n=a_{mn+1}=a_{mn}^{\delta_n}=H_n^{\delta_n}$ for any $n\geq 1$.
	The proof of $(C2)$ in~\S~\ref{proof} (or~\S~\ref{put}) shows that for any $q<H_{f+1}$ 
	we have
	\begin{equation} \label{eq:TRO}
	\Vert q\uz\Vert \geq d_{f}^{-1}+O(H_{f+1}d_{f+1}^{-1})= H_f^{-\delta_f}+O(H_{f+1}d_{f+1}^{-1}).
	\end{equation}
	We verify that the error term is of smaller order the main term.
	Indeed, $\delta_f=\gamma_2+o(1)$ as $f\to\infty$ implies $d_{f+1}=H_{f+1}^{\delta_{f+1}}=H_{f+1}^{\gamma_2+o(1)}=H_f^{m\gamma_2+o(1)}$, and since $m\gamma_2-m>\gamma_2$ by $\gamma_2>\gamma_1>(3+\sqrt{5})/2>1+1/(m-1)$ for $m\geq 2$, the 
	claim follows. Hence, combining \eqref{eq:TRO} with \eqref{eq:tauid}
	and using $\delta_f=\gamma_2+o(1)$ as $f\to\infty$ and \eqref{eq:leilei}, we conclude
	\[
	q < H_f^{ \frac{\gamma_2}{\gamma_1-1}+o(1) } < H_f^{ \gamma_1-1 }, \qquad f\geq f_0.
	\]
	On the other hand, by \eqref{eq:HG} for $n=f$ and Proposition~\ref{kp}, 
	we get the contradictory claim $q\gg H_f^{\gamma_1-1+\varepsilon}>H_f^{\gamma_1-1}$,
	unless
	$q=PH_f, P\in \mathbb{Z}\setminus\{0\}$ 
	is a multiple of $H_f=a_{mf}$. In fact, in the latter case if $\Vert H_f\zeta_1\Vert=|H_f\zeta_1-p_1|$ and $\Vert q\zeta_1\Vert=|q\zeta_1-r_1|$
	for integers $p_1, r_1$,
	then we must have $q/H_f=r_1/p_1=P$, i.e. $P(H_f,p_1)=(q,r_1)$. But then from \eqref{eq:HG} we get
	\[
	\Vert q \uz\Vert \geq \Vert q\zeta_1\Vert= P \Vert H_{f}\zeta_1\Vert \geq \Vert H_{f}\zeta_1\Vert\gg H_f^{-(\gamma_1-1)}\geq q^{ -(\gamma_1-1)},
	\]
	contradicting again \eqref{eq:tauid} for large $f$ (or equivalently $q$).
\end{proof}

Recall $\beta=\frac{1+\sqrt{5}}{2}$ and
let $\lambda\in (\beta,\infty)$ be given. Let
\[
\gamma_1= \lambda +1, \qquad \gamma_2\in (\lambda+1, \min\{ m\lambda, \lambda^2\})\neq \emptyset.
\]
Then \eqref{eq:dno}, \eqref{eq:leilei} hold.
Consider the derived sets $\mathcal{Q}_i$ and $\mathcal{Q}_1^{\ast}$, and
let $\mathcal{Q}:= \mathcal{Q}_1^{\ast}\times \prod_{i=2}^{m} \mathcal{Q}_i\subseteq \Rm$ for simplicity.
By Lemma~\ref{mitp} we have $\mathcal{Q}\subseteq \mathcal{W}_m(\lambda)$ 
and by Lemma~\ref{mure} and
since $\mathcal{Q}\subseteq \prod_{i=1}^{m} \mathcal{Q}_i$ we have
$\mathcal{Q}\subseteq
\textbf{F}_{m,c}$ again. It is further clear from the proof of Theorem~\ref{hdd} that for $\mathcal{Q}$ the bounds \eqref{eq:01}, \eqref{eq:02} still apply, as condition (iv) is metrically negligible. 
Hence, we get
a positive Hausdorff dimension of our set \eqref{eq:iset}.
For the asymptotical estimate as $\lambda\to\infty$, for $\gamma_2$
we may ignore the larger bound $\lambda^2$ and let $\gamma_2= m\lambda-o(1)$. Inserting in \eqref{eq:01} we observe that
for large $\lambda$ the right term in the minimum is smaller, and 
identifying main terms and estimating lower order terms gives 
$i/(\lambda m)-m^{-1}\lambda^{-1}(1+o(1))$ 
as a lower estimate for $3\leq i\leq m$, which sums up to $m/(2\lambda)-O(\lambda^{-1})$, 
the claimed asymptotical
bound. Theorem~\ref{beides} is proved.

\subsection{Generalizations}  \label{gen}
We sketch how to modify the construction of~\S~\ref{tem} to
get the refined claims on the exact order of 
ordinary approximation indicated below Theorem~\ref{beides}.
Let $\Psi(t)$ be any approximation function of decay
$o(t^{-\beta-\epsilon})$ as $t\to\infty$.
We alter (iv) from~\S~\ref{pbeides} by prescribing at step $n$
simultaneously the binary digits $g_{i,j}=g_j$ of all $\xi_i$, $1\leq i\leq m$,
at positions $j\in J_n:=\{ \lceil \gamma_1(n) h_n\rceil,
\lceil \gamma_1(n) h_n\rceil+1, \ldots, \lceil \gamma_1(n) h_n\rceil+n\}$,
where we put
\[
\gamma_1(n)=\left\lfloor \frac{\log \Psi(H_n)}{ \log H_n}\right\rfloor+1,\qquad 
\gamma_2(n)\in (\gamma_1(n), \min\{ m(\gamma_1(n)-1), (\gamma_1(n)-1)^2\})\neq \emptyset.
\] 
Thereby we obtain a subset $\tilde{\mathcal{Q}}$ of
$\prod \mathcal{Q}_i$ from~\S~\ref{putz} again, but with parameters
$\gamma_i$ depending on $n$.
Because of $|J_n|=n$, we can prescribe  
$\Vert H_n\uz\Vert/\Psi(H_n)$ up to a factor $1+2^{-n+1}$ 
at step $n$ by choosing digits $g_{i,j}$ suitably (mimicking the binary expansion of $\Psi(H_n)$ in $J_n$), so as $n\to\infty$
indeed we get a factor $1+o(1)$. 
The arguments from proof of Lemma~\ref{mitp}
further show that these values represent the local mimima 
of $\psi_{\uz}(Q)/\Psi(Q)$. Thus $\tilde{\mathcal{Q}}$ is contained in $\mathcal{W}_m(\Psi)\cap \textbf{F}_{m,c}$.
Moreover, the intervals $J_n$ are short enough not to
affect the asymptotics \eqref{eq:GS0}, \eqref{eq:GS} for every $n$
and $\gamma_{i}(n)$.
Given explicit lower and upper bounds for $-\log \Psi(t)/ \log t$,
we obtain intervals for $\gamma_i(n)$ uniformly for $n\ge 1$,
and may again infer metrical claims with some ``dynamical
variant'' of Lemma~\ref{ohlele}.
Assuming $\Psi(t)= t^{-\lambda+o(1)}$ 
for some $\lambda\in (\beta,\infty)$
as $t\to\infty$, the exact same
estimates \eqref{eq:01}, \eqref{eq:02} can be deduced and we
choose $\gamma_1= \lambda+1, \gamma_2=m\gamma_1-o(1)$ again
for optimization.
We leave the details to the reader.


Finally we sketch how to argue when $1<\lambda\le \beta$.
We let $\gamma_1=\lambda+1>2$ again and $\gamma_2>\gamma_1$ sufficiently
close to $\gamma_1$.
Then we fix the last coordinate
\[
\zeta_m= \sum_{n=1}^{\infty} 2^{-h_{n}} + \sum_{n=1}^{\infty} 2^{-\lfloor \gamma_1 h_n\rfloor },
\]
and take the other $\zeta_i\in\mathcal{Q}_i=\mathcal{Q}_i(\gamma_1,\gamma_2)$, $1\leq i\leq m-1$, as 
defined above with binary digit $0$ in intervals of types (i), (ii), (iii).
Note that $q=H_n=2^{h_n}$ and $q=2^{\lfloor \gamma_1 h_n\rfloor}\asymp H_n^{\gamma_1}$ induce small values
$\Vert q\zeta_m\Vert$. Conversely,
it can be deduced from the ``Folding Lemma'' by the same line of arguments
as in~\cite{bug2008} that we can only have
$\Vert q\zeta_m\Vert< q^{-(\gamma_1-1)-\epsilon}$ and hence
\eqref{eq:tauid}, if $q$ is a multiple 
of integers of these forms. To exclude these cases, we
can use a similar strategy as in the proof of $(C2)$ in~\S~\ref{proof} 
involving some case distinctions, assuming $\gamma_2>\gamma_1$ was chosen sufficiently close to $\gamma_1=\lambda+1$. For the Hausdorff dimensions
of our Cantor type sets of $\uz\in\textbf{F}_{m,c}\cap \mathcal{W}_m(\lambda)$
we get a lower bound 
$\dim_H(\prod_{i=1}^{m-1} \mathcal{Q}_i\times \{ \zeta_m\})\geq \sum_{i=1}^{m-1} \dim_H(\mathcal{Q}_i)>0$. We omit the technical
details again.

\section{Proof of Theorem~\ref{2}}

\subsection{The case of good $K$}  \label{6.1}
We restrict to $m\geq 3$, for $m=2$ we alter accordingly to \S~\ref{ss}.
If $K$ is good, then very similarly as in~\S~\ref{met} we 
can take $a_j=b^{c_j}$
for an increasing sequence of integers $c_j$ in the construction in \S~\ref{ko}, up to redefining $L_n=b^{\ell_n}$ as the smallest power
of $b$ so that $a_{mn+1}^{-1}< \Phi(Q^{\prime})$
for any $Q^{\prime}\leq b^{\ell_n}a_{mn+1}$. The analogue
of Proposition~\ref{uv} is easily checked as well in our setting.
The proof  of $(C3), (C1)$ 
works identically as for Theorem~\ref{F}. The proof of claim \eqref{eq:claim}
further works analogously
as Theorem~\ref{F} up to \eqref{eq:kkk}. Below, now 
we have to take $Q^{\prime}\leq bL_fa_{mf+m-1}$ in place of $Q^{\prime}\leq (L_f+1)a_{mf+m-1}$.
Hence
$\alpha_f\leq bL_f/L_f=b$. Then the decay
condition $(d3^{\prime}(b,R))$ yields a factor $R$ in place
of $1-\epsilon_2$ and (i) follows
with $\Omega=R$. 

For (ii), first recall that we may relax assumption \eqref{eq:h} 
to $a_{mn+1}$ being
just a large enough multiple, in place
of power, of $a_{mn}$.
This corresponds to $c_{mn}<c_{mn+1}$ in place of $c_{mn}|c_{mn+1}$, so
that we can choose $A, B$ in $(D(b))$ freely. 
As remarked above, an according variant of
Proposition~\ref{uv} holds in our setting $a_j=b^{c_j}$ as well.
Then with large $A,B$ as in $(D(b))$
we may choose $c_{mn+1}=A_n=A$ and $c_{mn+m-1}+\ell_n=B_n=B$, i.e.
$\ell_n=B_n-A_n$, in step $n$ 
so that the quotient $\Phi(b^B)/b^{-A}>1$ is arbitrarily
close to $1$ again. Hence we may choose $\alpha_f=Q^{\prime}/(L_fa_{mf+m-1})>1$ arbitrarily close to $1$ again, and 
consequently get the same result as in Theorem~\ref{F}.

The bound on the packing dimension for $K=C_{b,\{0,1\}}^m$
follows similarly as for $\Rm$. We
now instead have that the sumset $(K\cap \mathcal{V}) + (K\cap \mathcal{S})$ with $\mathcal{V}=\mathcal{V}_{m}^{(b)}(\mu-\varepsilon)$ 
as in~\S~\ref{met} but for general $b\geq 2$ 
contains $K$, and we conclude with
Tricot's estimate again. Here we use the general version of~\cite[Lemma~5.6]{ichneu} to bound the Hausdorff
dimension of $K\cap \mathcal{V}$ from above. See also the very similar proof of~\cite[Theorem~4.1]{ichneu}. A positive lower
bound for the Hausdorff dimension follows from very similar arguments as in~\S~\ref{putz}. We need to replace $P_n$ in \eqref{eq:GS0} resp. \eqref{eq:GS} by 
$P_n\asymp H_n^{(\gamma_2-\gamma_1)\log 2/\log b}$ 
resp. $P_{2n}\asymp H_n^{(\gamma_2-\gamma_1)\log 2/\log b}$ and 
$P_{2n+1}\asymp H_n^{(i-2)\gamma_2\log 2/\log b}$ 
and keep $\epsilon_n$ unchanged, we omit the slightly cumbersome calculation.

\subsection{General case}

We may assume $|W_i|=2$ for $1\leq i\leq m$. Then we have the identity of sets
\begin{equation}  \label{eq:cant}
C_{b,W_i}= w_{i,1}C_{b,\{0,1\}}+ \frac{w_{i,2}}{b-1} 
\end{equation}
for $w_{i,1}\geq 1, w_{i,2}\geq 0$ integers with $w_{i,1}+w_{i,2}\leq b-1$, where $W_i=\{w_{i,2}, w_{i,1}+w_{i,2}\}$. We use this 
identity to reduce the general case to the special case 
of good $K$.

Let $\Phi$ be any function satisfying $(d1), (d2), (d3^{\prime}(b,R))$ 
for some $R\in (0,1)$. From $(d1)$ we get the weaker condition
\[
\Phi(t)< (b-1)^{2}R^{-1} \cdot t^{-1/m}, \qquad t\geq t_0.
\]
Define the auxiliary function
\[
\tilde{\Phi}(t)= (b-1)^{-2}R\cdot \Phi(t),
\]
which by construction satisfies $(d1), (d2), (d3^{\prime}(b,R))$, the assumptions
of Theorem~\ref{2}.
Start with $Q^{\ast}>1$ an arbitrary, large number and let $Q=Q^{\ast}/b$.
Now, in \S~\ref{6.1} we showed there exists
$\ut=(\theta_1,\ldots,\theta_m)\in \tilde{K}:= C_{b,\{0,1\}}^m$ that
satisfies (i) of Theorem~\ref{2} with respect to the good $\tilde{K}$
and $\tilde{\Phi}$. In particular the system
\begin{equation}  \label{eq:hurra}
|q\theta_i-p_i| \leq \tilde{\Phi}(Q), \qquad\qquad 1\leq i\leq m,\; 1\leq q\leq Q,
\end{equation}
has a solution in integers $q, p_i$. Then, according to \eqref{eq:cant}, write
\begin{equation}  \label{eq:ff}
\xi_i= w_{i,1} \theta_i + \frac{w_{i,2}}{b-1}, \qquad\qquad 1\leq i\leq m,
\end{equation}
so that $\underline{\xi}=(\xi_1,\ldots, \xi_m)\in K$. 
We claim it satisfies the properties $(C1), (C2), (C3)$ with respect to $\Phi$. 
From \eqref{eq:hurra} we see that with
\[
q^{\ast}= (b-1)q, \qquad p_i^{\ast}= (b-1)p_i + qw_{i,2}
\]
we have
\[
|q^{\ast} \xi_i - p_i^{\ast} | \leq (b-1) w_{i,1} \tilde{\Phi}(Q)\leq (b-1)^2 \tilde{\Phi}(Q), \qquad 1\leq i\leq m.
\]
Since $Q^{\ast}= (b-1)Q$, from $(d3^{\prime}(b,R))$ applied
to $Q^{\ast}/b$ and since
$\tilde{\Phi}$ decays we see
\[
\tilde{\Phi}(Q)= \tilde{\Phi}(Q^{\ast}/(b-1)) \leq \tilde{\Phi}(Q^{\ast}/b)\leq R^{-1}\tilde{\Phi}(Q^{\ast})
\]
and combining we have a solution to
\[
1\leq q^{\ast}\leq Q^{\ast},\qquad \Vert q^{\ast}\underline{\xi}\Vert \leq (b-1)^2R^{-1} \tilde{\Phi}(Q^{\ast})= \Phi(Q^{\ast}).
\]
Note that $Q^{\ast}>1$ was arbitrary, so $(C1)$ holds
for $\underline{\xi}\in K$. 

Now define another function $\Psi(t)= v\tilde{\Phi}(t)$ for $v>0$ be to determined later. Assume conversely that $\underline{\xi}$ as 
in \eqref{eq:ff} satisfies
\begin{equation}  \label{eq:fal}
1\leq q^{\ast} \leq Q^{\ast}, \qquad
|q^{\ast} \xi_i - p_i^{\ast} | \leq \Psi(Q^{\ast}), \qquad 1\leq i\leq m,
\end{equation}
holds for some large $Q^{\ast}$. 
From \eqref{eq:ff} we see
that for
\begin{equation}  \label{eq:sit}
q_i= (b-1)w_{i,1}q^{\ast}, \qquad p_i= (b-1)p_i^{\ast}-q^{\ast}w_{i,2}
\end{equation}
we have
\begin{equation}  \label{eq:prei}
|q_i \xi_i- p_i| \leq (b-1)\Psi(Q^{\ast}), \qquad 1\leq i\leq m.
\end{equation}
We take $q$ the lowest common multiple of the $q_i$, for which
since $w_{i,1}\leq b-1$ an estimate is given by
\begin{equation}  \label{eq:eip}
q\leq \Gamma_{b}(b-1)q^{\ast}, \qquad \Gamma_{b}= lcm(1,2,\ldots,b-1).
\end{equation}
Note that we may sharpen this by taking the maximum least common 
multiple of any $m$ integers at most $b-1$, as claimed in Remark~\ref{reh}.
Now for $\tilde{p}_i= (q/q_i)p_i\in\mathbb{Z}$ and $Q= (b-1)\Gamma_{b} Q^{\ast}$, we have
$1\leq q\leq Q$ and combining \eqref{eq:prei}, \eqref{eq:eip} and \eqref{eq:sit} we infer
\begin{align*}
|q\theta_i- \tilde{p}_i| &= \frac{q}{q_i}|q_i \theta_i- p_i| \leq \frac{q}{q_i} (b-1)\Psi(Q^{\ast})\leq
\Gamma_{b}(b-1)^2\frac{q^{\ast}}{q_i}\Psi(Q^{\ast})\\
&\leq
\Gamma_{b}(b-1)w_{i,1}^{-1}\Psi(Q^{\ast})
\leq \Gamma_{b}(b-1)\Psi(Q^{\ast}).
\end{align*}
Now since $\Psi$ satisfies $(d3^{\prime}(b,R))$ as well, 
estimating trivially
$(b-1)\Gamma_{b}\leq b^{ b}$ (see Remark~\ref{reh} for 
asymptotical improvements) and repeated application 
and monotonicity of $\Psi$ shows 
\[
\Psi(Q) \geq \Psi(b^{b} Q^{\ast}) \geq  R^{b} \Psi(Q^{\ast}).
\]
Inserting, for $1\leq i\leq m$ we have
\begin{align*}
1\leq q\leq Q, \qquad
|q\theta_i- \tilde{p}_i| \leq 
(\Gamma_{b}(b-1)R^{-b})\Psi(Q)=
(\Gamma_{b}(b-1)R^{-b}v)\tilde{\Phi}(Q),
\end{align*}
in other words
\[
1\leq q\leq Q, \qquad \Vert q\ut\Vert <(\Gamma_{b}(b-1)R^{-b}v)\cdot \tilde{\Phi}(Q).
\]
On the other hand, by assumption 
our $\ut\in \tilde{K}$ satisfies
$(C2^{\prime})$ of Theorem~\ref{2} with respect to $\tilde{\Phi}$ and $\Omega=R$. In other words, for certain $Q^{\ast}$ and the induced $Q=\Gamma_{b} Q^{\ast}$, we have $\Vert q\ut\Vert>R\tilde{\Phi}(Q)$ for all $1\leq q\leq Q$.
If $v<\Gamma_{b}^{-1}(b-1)^{-1}R^{b+1}$ we
get a contradiction. Hence the assumption \eqref{eq:fal} was false and
for these $Q^{\ast}, v$ and for all $1\leq q^{\ast}\leq Q^{\ast}$
we have
\[
\Vert q^{\ast} \underline{\xi}\Vert > \Psi(Q^{\ast}) = v\tilde{\Phi}(Q^{\ast})= \frac{R}{(b-1)^2}v\Phi(Q^{\ast}),
\]
for any $1\leq q^{\ast}\leq Q^{\ast}$. Inserting the bound for $v$,
this means $\underline{\xi}$ satisfies $(C2^{\prime})$ for
\[
\Omega(b,R)= \Gamma_{b}^{-1}(b-1)^{-1}R^{b+1}\cdot \frac{R}{(b-1)^2}= 
\frac{R^{b+2} }{(b-1)^3\Gamma_{b}}.
\]
Finally $(C3)$ for $\underline{\xi}$ clearly follows from \eqref{eq:ff} as well when considering integers $(b-1)q$ for $q$
inducing the same property $(C3)$ for $\ut$.

\section{Proof of Corollary~\ref{ok} and Corollary~\ref{rok}} \label{8}

We verify property $(d3^{\prime}(b,R))$ for $R=b^{-1/m}$
for the Cantor set setting. 

\begin{lemma}  \label{lemuren}
	Let $m\geq 2, b\geq 2$ be integers and $c\in(0,1)$. Then for any 
	$\epsilon>0$ 
	there exist arbitrarily
	large integers $A,B$ with
	\[
	cb^{-B/m}b^{-1/m}\leq b^{-A} < cb^{-B/m}.
	\]
\end{lemma}

\begin{proof}
	It suffices to take $B$ the largest integer with $b^{-A}< cb^{-B/m}$.
	By maximility of $B$ the other inequality holds as well. 
	\end{proof}

The lemma states that $\Phi(q)=cq^{-1/m}$ satisfies $(d3^{\prime}(b,R))$ for any pair $(b,R)$
with an integer $b\geq 2$ and $R=b^{-1/m}$, and Corollary~\ref{ok}
follows from part (i) of Theorem~\ref{2}. 

\begin{lemma}  \label{lemur}
	Let $m\geq 2, b\geq 2$ be integers, $c\in(0,1)$ and
	$\tau>0$ irrational. Then for any 
	$\epsilon>0$ 
	there exist arbitrarily
	large integers $A,B$ with
	\[
	(1-\epsilon)cb^{-B\tau}<b^{-A}\leq cb^{-B\tau}.
	\]
\end{lemma}

\begin{proof}
	Taking logarithms the claim becomes
	\[
	0 <  (A-B\tau)\log b - \log c < -\log (1-\epsilon).
	\]
	Since $\tau$ is irrational,
	the set of values of $A-\tau B$ and thus also of
	 $(A-B\tau)\log b-\log c$ when taking all integer pairs 
	$A,B$ are dense in $\mathbb{R}$. The claim follows.
\end{proof}

From the lemma we see that $\Phi(q)=cq^{-\tau}$ for $\tau$ irrational satisfies $(D(b))$ for any integer
$b\geq 2$, and part (ii) of Theorem~\ref{2} implies Corollary~\ref{rok}.

Assume we are given an effective upper bound for the irrationality exponent of $\tau$. Then,
using a result on uniform inhomogeneous approximation 
due to Bugeaud and Laurent~\cite{buglau},
we could state upper bounds for the smallest $A,B$ satisfying the hypothesis of Lemma~\ref{lemur} in terms of $\epsilon$, independent from $c$. Similar to the remark below
Theorem~\ref{F}, this in turn implies an effective rate  
at which we can let $\varepsilon\to 0$ in terms of $Q$
in Corollary~\ref{rok}, if we admit a factor $1+\varepsilon$
in condition $(C1)$. We do not make this explicit here.

\section{Appendix: The linear form problem}

Let $\scp{.,.}$ be the standard scalar product on $\Rm$ and for $\underline{y}=(y_1,\ldots,y_m)\in \Rm$ denote by
$|\uy|_{\infty}= \max_{1\leq i\leq m} |y_i|$ the maximum
norm. Recall further the notation $\Vert .\Vert$ introduced in \S~\ref{s1.1}.
For $c^{\ast}\in(0,1]$, let $Di_m^{\ast}(c^{\ast})$ be the set of $\ux\in\Rm$ for which the system
\begin{equation}  \label{eq:linF}
0<|\uy|_{\infty} \leq Q^{\ast}, \qquad
\Vert \scp{\uy,\ux}\Vert \leq c^{\ast}Q^{\ast -m}
\end{equation}
has a solution in an integer vector $\uy$, for all large 
parameters $Q^{\ast}$.
By a variant of Dirichlet's Theorem we have $Di_m^{\ast}(1)=\Rm$ for any $m\geq 1$. Let again $Di_m^{\ast}= \cup_{c<1} Di_m^{\ast}(c)$.
Let $Bad_m^{\ast}$ be the set of badly approximable 
linear forms, its defining property being
that for some $c^{\ast}>0$ and all $Q^{\ast}$ there is no integer vector solution to \eqref{eq:linF}. For completeness further define accordingly
$Sing_m^{\ast}= \cap_{c>0} Di_m^{\ast}(c)$ and $\boldsymbol{FS}_m^{\ast}= Di_m^{\ast}\setminus (Bad_m^{\ast}\cup Sing_m^{\ast})$. We point out the well-known identities
\[
Di_m=Di_m^{\ast} , \qquad Sing_m=Sing_m^{\ast}, \qquad Bad_m=Bad_m^{\ast}, \qquad \boldsymbol{FS}_m=\boldsymbol{FS}_m^{\ast}.
\]
Note however that the sets $Di_m(c)$ and $Di_m^{\ast}(c)$ do not coincide
for the same parameter $c<1$.
We expect analogous
results to~\S~\ref{se2.2}, in particular counterparts 
of Theorem~\ref{A} and Theorem~\ref{hdd}, 
for the linear form setting.
From Corollary~\ref{c} and a transference result phrased below, we obtain
the following more modest claim.

\begin{theorem} \label{Fr}
	For $m\geq 2$ an integer and any $c^{\ast}\in(0,1]$, 
	if we let 
	\[
	\omega=\omega(m,c^{\ast}):= (m+1)^{-m^2-m}(c^{\ast})^{m^2} \in (0,c^{\ast}), 
	\]
	then the set
	\[
	Di_m^{\ast}(c^{\ast})\setminus (Di_m^{\ast}(\omega) \cup Bad_m^{\ast})
	\subseteq \boldsymbol{FS}_m^{\ast} 
	\]
	has packing dimension at least $m-1$ and Hausdorff dimension
	at least as in Theorem~\ref{hdd}.
\end{theorem}

As $m\to\infty$, the value $\omega(m,c^{\ast})$ asymptotically satisfies
\[
\omega(m,c^{\ast}) > (c^{\ast})^{m^2}e^{-m^2 \log m-o(m^2 \log m)}. 
\]
Our result 
should be compared with the following partial claim of~\cite[Theorem~1.5]{beretc} (see
also~\cite{marnat})
obtained from a very different, unconstructive method.

\begin{theorem}[Beresnevich, Guan, Marnat, Ram\'irez, Velani] \label{bt}
	Let $m\geq 2$ an integer and 
	\[
	\kappa_m = e^{-20(m+1)^3(m+10)}.
	\]
	Then the set 
	\[
	Di_m^{\ast}(c^{\ast})\setminus (Di_m^{\ast}(\kappa_mc^{\ast}) \cup Bad_m)
	\subseteq \boldsymbol{FS}_m^{\ast}
	\]
	is uncountable.
\end{theorem}

We should remark
that in the statement we suppressed some more information on other
exponents of approximation given in~\cite[Theorem~1.5]{beretc}.
Moreover, the exact shape of the polynomial in the exponent of $\kappa_m$, in particular the leading coefficient 20, can be readily optimized with sharper estimates at certain places in~\cite{beretc}.
Note that in contrast to our result, the value $\kappa_m$ 
in Theorem~\ref{bt} is independent from $c^{\ast}$.
We see that for large $m$ and large enough $c^{\ast}\in(0,1]$,
roughly  as soon as $c^{\ast } > e^{-m^2}$, our 
bound from Theorem~\ref{Fr} is stronger. For $c^{\ast}$ very 
close to $1$,
we basically can reduce the quartic polynomial in $m$ 
within the exponent in $\kappa_m$ to a quadratic polynomial.

For the deduction of Theorem~\ref{Fr} it is convenient to apply a transference result by
German~\cite{german} based on geometry of numbers, which
improves on previous results by Mahler. Concretely,
we use the following special cases of~\cite[Theorem~7]{german}, where we implicitly include the upper estimate
$\Delta_d^{-1}\leq d^{1/2}$ from~\cite[\S~2]{german}
for the quantity $\Delta_d$ defined there, where $d=m+1$
in our situation.

\begin{theorem}[German]  \label{og}
Let $m\geq 1$ and $\ux\in\Rm$. Let $X, U$ be positive parameters. 
\begin{itemize}
	\item[(i)]  Let $x\in \mathbb{Z}$ and assume 
	\[
    0<|x|\leq X, \qquad \Vert \ux x\Vert \leq U.	
	\]
	Then there exists $\uy^{\ast}\in\mathbb{Z}^m$
	so that
	\[
	0< |\uy^{\ast}|_{\infty}\leq Y, \qquad \Vert\scp{\uy^{\ast},\ux}\Vert\leq V,
	\] 
	where 
	\[
	Y= (m+1)^{1/(2m)}X^{1/m}, \qquad V= (m+1)^{1/(2m)}X^{1/m-1}U.
	\]
	\item[(ii)] Let $\uy\in\mathbb{Z}^m$ and assume 
	\[
	0<|\uy|_{\infty}\leq X, \qquad \Vert\scp{\uy,\ux}\Vert\leq U.
	\]
	Then there exists $x^{\prime}\in\mathbb{Z}$
	so that
	\[
	0< |x^{\prime}|\leq Y^{\prime}, \qquad \Vert x^{\prime}\ux\Vert \leq V^{\prime}
	\] 	
	where 
	\[
	Y^{\prime}= (m+1)^{1/(2m)}XU^{1/m-1}, \qquad V^{\prime}= (m+1)^{1/(2m)}U^{1/m}.
	\]
\end{itemize}

\end{theorem}

We now prove our claim.

\begin{proof}[Proof of Theorem~\ref{Fr}]
	Let $c\in(0,1]$ to be chosen later. By Corollary~\ref{c} and Theorem~\ref{hdd},
	the set $Di_m(c)\setminus (\cup_{\epsilon>0} Di_m(c-\epsilon) \cup Bad_m)=Di_m(c)\setminus (\cup_{\epsilon>0} Di_m(c-\epsilon) \cup Bad_m^{\ast})$ has the stated metrical properties. 
	Take any $\ux$ in this set. 
	Take arbitrary, large $Y$ and put $X=(Y(m+1)^{-1/(2m)})^m$, which
	is also large.
	Now the hypothesis of (i) from Theorem~\ref{og}
	holds when we take
	\[
	U=cX^{-1/m}.
	\]
	From the conclusion we get $\uy^{\ast}\in \mathbb{Z}^m$ that satisfies
	\[
	0< |\uy^{\ast}|_{\infty}\leq Y=(m+1)^{1/(2m)}X^{1/m}
	\]
	and
	\[
	 \Vert\scp{\uy^{\ast},\ux}\Vert\leq V= (m+1)^{1/(2m)}X^{1/m-1}U=
	c(m+1)^{1/(2m)}X^{-1}= c^{\ast}Y^{-m},
	\]
	where
	\begin{equation}  \label{eq:kloa}
	c^{\ast}= c(m+1)^{1/2+1/(2m)}.
	\end{equation}
	Observe this holds for all large $Y$, so $\ux\in Di_m^{\ast}(c^{\ast})$.
	
	Now assume that for some $\tilde{c}\in(0,1)$ we have
	$\ux\in Di_m^{\ast}(\tilde{c})$. That means for all large $X$ we may take
	\[
	U= \tilde{c} X^{-m}
	\]
	and the hypothesis of (ii) from Theorem~\ref{og} is satisfied
	for some $\uy\in\mathbb{Z}^m$.
	From the conclusion we get the existence of a positive integer $x^{\prime}$
	satisfying
	\[
	0< |x^{\prime}|\leq Y^{\prime}= (m+1)^{1/(2m)}XU^{1/m-1}=  
	(m+1)^{1/(2m)} \tilde{c}^{1/m-1} X^m
	\]
	and
		\[
	 \Vert x^{\prime}\ux\Vert \leq V^{\prime}= (m+1)^{1/(2m)}U^{1/m}=(m+1)^{1/(2m)} \tilde{c}^{1/m} X^{-1}= \tilde{C}Y^{\prime -1/m},
	\]
	where
	\begin{equation} \label{eq:ocker}
	\tilde{C}= (m+1)^{1/2+1/(2m)} \tilde{c}^{1/m^2}.
	\end{equation}
	 Now, since
	$Y^{\prime}$ can be any large number by choosing $X$ suitably,
	we infer $\ux\in Di_m(\tilde{C})$. 
	If $\tilde{C}< c$
	and assuming $c\leq 1$, this contradicts our choice of $\ux$. By \eqref{eq:kloa}, the 
	latter condition $c\leq 1$ clearly holds as soon as $c^{\ast}\leq 1$,
	so we require $\tilde{C}\geq c$, which by \eqref{eq:ocker} 
	and \eqref{eq:kloa} leads to
	\[
	\tilde{c}\geq 
	((m+1)^{-1/2-1/(2m)}c)^{m^2}= (m+1)^{-m^2/2-m/2} c^{m^2}= 
	(m+1)^{-m^2-m}c^{\ast m^2}.
	\]
	Combining our results we see that $\ux\in Di_m^{\ast}(c^{\ast})\setminus (Di_m^{\ast}((m+1)^{-m^2-m}c^{\ast m^2})\cup Bad_m^{\ast})$.
	\end{proof}


It may be possible to derive similar, possibly stronger, effective results when combining Corollary~\ref{c}
	with the essential method of Davenport and Schmidt~\cite{ds},
	however in reverse direction (we need the conclusion from simultaneous approximation to linear form instead of the other way round), instead of~\cite{german}. We also want to refer to~\cite[\S~4]{beretc}, in
	particular~\cite[Lemma~4.9]{beretc}, in this context.

\vspace{0.55cm}

{\em The author thanks Mumtaz Hussain for pointing out
Example 4.6 in Falconer's book as a tool to estimate
the Hausdorff dimension in Theorems~\ref{hdd},~\ref{beides}. }


\begin{thebibliography}{99}
	
	
		\bibitem{akm} R. K. Akhunzhanov. Vectors of a given Diophantine type. II. (Russian. Russian summary)
	{\em Mat. Sb.} 204 (2013), no. 4, 3--24; translation in {\em Sb. Math.} 204 (2013), no. 3-4,
	463--484.
	
	
	\bibitem{am2} R. K. Akhunzhanov, N. G. Moshchevitin.
	A note on Dirichlet spectrum. {\em arXiv: 2111.12259}.
	

	
	\bibitem{as} R. K. Akhunzhanov, D. O. Shatskov. On Dirichlet spectrum for two-dimensional simultaneous 
	Diophantine approximation. {\em Mosc. J. Comb. Number Theory} 3 (2013), no. 3--4, 5--23.



	
\bibitem{beretc} V. Beresnevich, L. Guan, A. Marnat, F. Ram\'irez, S. Velani. Dirichlet is not just Bad
and Singular. {\em arXiv: 2008.04043}.


\bibitem{bug2008} Y. Bugeaud. Diophantine approximation and Cantor sets. {\em Math. Ann.} 341 (2008), no. 3, 677-–684.


\bibitem{buglau} Y. Bugeaud, M. Laurent. On exponents of homogeneous and inhomogeneous Diophantine approximation. {\em Mosc. Math. J.} 5 (2005), no. 4, 747--766, 972.



\bibitem{cheche} Y. Cheung, N. Chevallier. Hausdorff dimension of singular vectors. {\em Duke Math. J.} 165 (2016), no. 12,
2273--2329.

\bibitem{dfsu1} T. Das, L. Fishman, D. Simmons, M. Urba\'nski. A variational principle in the parametric geometry of numbers, with applications to metric Diophantine approximation. {\em C. R. Math. Acad. Sci. Paris} 355 (2017), no. 8, 835--846.

\bibitem{dfsu2} T. Das, L. Fishman, D. Simmons, M. Urba\'nski. A variational principle in the parametric geometry of numbers. 
{\em arXiv: 1901.06602}.

\bibitem{ds} H. Davenport, H., W. M. Schmidt. Dirichlet's theorem on diophantine approximation. II. {\em Acta Arith.} 16 (1969/70), 413–-424.

\bibitem{ds2} H. Davenport, W. M. Schmidt. Dirichlet's theorem on diophantine approximation. 1970 Symposia Mathematica, Vol. IV (INDAM, Rome, 1968/69) pp. 113--132.

	\bibitem{falconer} K. Falconer.
Fractal geometry.
Mathematical foundations and applications. John Wiley \& Sons, Ltd., Chichester, 1990. xxii+288 pp. 

\bibitem{feng} J.  Feng, J. Xu. Sets of Dirichlet non-improvable numbers with certain order in the theory of continued fractions. {\em Nonlinearity} 34 (2021), no. 3, 1598--1611.

\bibitem{german} O. German. On Diophantine exponents and Khintchine's transference principle. {\em Mosc. J. Comb. Number Theory} 2 (2012), no. 2, 22–-51.

\bibitem{huang} J. Huang, J. Wu. Uniformly non-improvable Dirichlet set via continued fractions. {\em Proc. Amer. Math. Soc.} 147 (2019), no. 11, 4617--4624.

\bibitem{huss} M. Hussain, D. Y. Kleinbock,  
N. Wadleigh, B.-W. Wang. Hausdorff measure of sets of Dirichlet non-improvable numbers. {\em Mathematika} 64 (2018), no. 2, 502--518.







\bibitem{jarnik} 
V. Jarn\'ik. \"Uber die simultanen diophantischen Approximationen. {\em Math. Z.} 33 (1931), 505--543.

\bibitem{khint} A. Y. Khintchine. \"Uber eine Klasse linearer diophantischer Approximationen. {\em Rend. Circ. Mat. Palermo} 50 (1926), 706--714. 

\bibitem{km}  D. Y. Kleinbock, S. Mirzadeh. On the dimension drop conjecture for diagonal flows on the space of lattices. {\em arXiv:} 2010.14065. 

\bibitem{kr} D. Y. Kleinbock, A. Rao. Abundance of Dirichlet improvable
pairs with respect to arbitrary norm. {\em arXiv: 2107.10298}.

\bibitem{kr2} D. Y. Kleinbock, A. Rao. Weighted uniform Diophantine approximation of systems of linear forms. {\em arXiv: 2111.07115}.

\bibitem{ksw}  D. Y. Kleinbock, A. Str\"ombergsson, S. Yu. A measure estimate in geometry of numbers and improvements to Dirichlet's theorem.
{\em arXiv: 2108.04638}.

\bibitem{kw} D. Y. Kleinbock, N. Wadleigh. A zero-one law for improvements to Dirichlet's Theorem. {\em Proc.
Amer. Math. Soc.} 146 (2018), no. 5,  1833--1844.




\bibitem{marnat} A. Marnat. Diophantine 
sets and Dirichlet improvability. {\em arXiv: 2201.11093}.

	
	\bibitem{roy} D. Roy. On Schmidt and Summerer parametric geometry of numbers. {\em Ann. of Math.} (2) 182 (2015), no. 2, 739--786.
	

	
\bibitem{ichostrava}  J. Schleischitz. Diophantine approximation and special Liouville numbers. {\em Commun. Math.} 21 (2013), no. 1, 39--76.
	
	\bibitem{ichmjnt} J. Schleischitz. Generalizations of a result of Jarn\'ik on simultaneous approximation. {\em Mosc. J. Comb. Number Theory} 6 (2016), no. 2-3, 253--287.
	
		\bibitem{ichdimproduct} J. Schleischitz. Cartesian product sets in Diophantine approximation with large Hausdorff dimension. {\em arXiv: 2002.08228}.
	
	\bibitem{ichneu} J. Schleischitz. Metric results on inhomogeneously singular vectors. {\em arXiv: 2201.01527}.
	
	\bibitem{ss} W. M. Schmidt, L. Summerer. Parametric geometry of numbers and applications. {\em Acta Arith.} 140 (2009), no. 1, 67--91. 
	
		\bibitem{tricot} C. Tricot Jr. Two definitions of fractional dimension. {\em Math. Proc. Cambridge Philos. Soc.} 91 (1982), no. 1, 57--74.



\end{thebibliography}
\end{document}